\documentclass[11pt,a4paper]{amsart}
\usepackage[T1]{fontenc}
\usepackage{lmodern,amsmath,amssymb,amsthm,mathtools}
\usepackage[margin=1in]{geometry}
\usepackage[expansion=false]{microtype}
\usepackage[hidelinks]{hyperref}
\usepackage{bookmark,needspace}
\allowdisplaybreaks[2]
\numberwithin{equation}{section}
\newtheorem{theorem}{Theorem}[section]
\newtheorem{proposition}[theorem]{Proposition}
\newtheorem{lemma}[theorem]{Lemma}

\theoremstyle{definition}
\newtheorem{definition}[theorem]{Definition}
\newtheorem{remark}[theorem]{Remark}
\newcommand{\R}{\mathbb R}
\newcommand{\eps}{\varepsilon}
\newcommand{\dd}{\,\mathrm d}
\newcommand{\ii}{\mathrm i}
\newcommand{\Imn}{\operatorname{Im}}
\newcommand{\Ren}{\operatorname{Re}}
\newcommand{\supp}{\operatorname{supp}}
\newcommand{\Div}{\operatorname{div}}
\newcommand{\BB}{\mathcal B}
\newcommand{\Kcal}{\mathcal K}
\newcommand{\cK}{K}
\newcommand{\cN}{\mathcal N}
\newcommand{\tr}{\operatorname{tr}}
\newcommand{\M}{\mathcal M}

\title[Scattering below the ground state]
{Scattering below the ground state for the three-dimensional focusing energy-critical NLS}
\author{Qingtang Su}
\address{Qingtang Su
\newline \indent Academy of Mathematics and Systems Science, Chinese Academy of Sciences, Beijing, China.
\newline \indent Morningside Center of Mathematics, Beijing, China.}
\email{suqingtang@amss.ac.cn}

\author{Zehua Zhao}
\address{Zehua Zhao
\newline \indent School of Mathematics and Statistics, Beijing Institute of Technology,
\newline \indent Key Laboratory of Algebraic Lie Theory and Analysis, Ministry of Education,
Beijing 100081, China}
\email{zzh@bit.edu.cn}
\subjclass[2020]{35Q55, 35B40, 35B45}
\keywords{nonlinear Schr\"odinger equation,
global well-posedness, scattering, ground state, concentration compactness,
interaction Morawetz estimates}
\hypersetup{pdftitle={Scattering below the ground state for the three-dimensional focusing energy-critical NLS},
pdfauthor={Qingtang Su and Zehua Zhao},
pdfsubject={Scattering below the ground state for the three-dimensional focusing energy-critical NLS}}
\begin{document}
\raggedbottom
\begin{abstract}
We prove global well-posedness and scattering for the three-dimensional
focusing energy-critical nonlinear Schr\"odinger equation with initial
data below the ground-state energy and gradient thresholds. This extends
the seminal result of Kenig and Merle (Invent. Math., 2006) to general
data in $\dot H^1(\mathbb R^3)$. Classical concentration compactness and
the exclusion of rapid frequency cascades reduce the problem to the
quasi-soliton scenario. The main ingredient is a frequency-localized
interaction Morawetz estimate based on an explicit convex weight whose
time-dependent radius is determined implicitly by the mass density of
the high-frequency component. The ground-state gap controls the focusing
contribution with constants independent of the mass of this component.
The radius is chosen to retain quartic spacetime control, while the
continuity equation expresses the moving-radius contribution as a
nonnegative square, up to frequency-truncation errors. We absorb these
errors by closing the Morawetz estimate together with an auxiliary
spacetime bound for the nonlinearity. The resulting quartic bound
excludes the quasi-soliton scenario and completes the proof.
\end{abstract}

\maketitle
\tableofcontents
\section{Introduction}\label{sec:intro}

\subsection{The main result}\label{subsec:intro-main}
We consider the three-dimensional focusing energy-critical nonlinear
Schr\"odinger equation
\begin{equation}\label{eq:nls}
 (\ii\partial_t+\Delta)u=-|u|^4u,
 \qquad u(0)=u_0\in\dot H^1(\R^3).
\end{equation}
The scaling
$u_\lambda(t,x)\coloneqq\lambda^{\frac{1}{2}}u(\lambda^2t,\lambda x)$
preserves both the equation and the $\dot H^1(\R^3)$ norm.

A strong solution on an interval $I$ containing $0$ belongs to
$C(I;\dot H^1(\R^3))$, lies in $L^{10}_{t,x}(J\times\R^3)$ for
every compact $J\subset I$, and satisfies
\[
 u(t)=e^{\ii t\Delta}u_0
      +\ii\int_0^t e^{\ii(t-s)\Delta}|u(s)|^4u(s)\,ds.
\]
Cazenave and Weissler \cite{CW} established local well-posedness in
$\dot H^1(\R^3)$. The continuation and scattering criteria reduce
global existence and scattering to a finite $L^{10}_{t,x}$ norm on
the maximal lifespan; see \cite[Section~2]{KM}. Strichartz estimates
give this bound for sufficiently small data. For general background,
see \cite{CazenaveBook,DodsonBook,KVNotes,TaoBook}.

The stationary solution
\[
 W(x)\coloneqq (1+\frac{|x|^2}{3})^{-\frac{1}{2}},\qquad -\Delta W=W^5,
\]
precludes scattering for arbitrary data. The conserved energy is
$E(v)\coloneqq\frac{1}{2}\|\nabla v\|_2^2-\frac{1}{6}\|v\|_6^6$.
The ground state $W$ attains equality in the sharp Sobolev inequality
\cite{Aubin,Talenti}. With
$G\coloneqq\|\nabla W\|_2^2=\|W\|_6^6$, we have
\begin{equation}\label{eq:sharp}
 \|f\|_6^6\le G^{-2}\|\nabla f\|_2^6,
 \qquad E(W)=\frac{G}{3}.
\end{equation}
The below-ground-state scattering problem asks whether data below
both the energy and gradient thresholds of $W$ generate global
scattering solutions; see \cite{KM,KVfocus}. We prove this in
dimension three without radial symmetry or a finite-mass assumption.

\Needspace{12\baselineskip}
\begin{theorem}\label{thm:main}
Suppose
\begin{equation}\label{eq:threshold}
 u_0\in\dot H^1(\R^3),\qquad E(u_0)<E(W),\qquad
 \|\nabla u_0\|_2<\|\nabla W\|_2.
\end{equation}
Then the unique maximal strong solution of \eqref{eq:nls} is global,
belongs to
\[
 C(\R;\dot H^1)\cap L^{10}_{t,x}(\R\times\R^3),
\]
and scatters in both time directions: there exist
$u_\pm\in\dot H^1(\R^3)$ such that
\[
 \lim_{t\to\pm\infty}\|u(t)-e^{\ii t\Delta}u_\pm\|_{\dot H^1}=0.
\]
\end{theorem}

\begin{remark}\label{rem:decay}
If, in addition, $u_0\in L^1(\R^3)$, then the solution in
Theorem~\ref{thm:main} satisfies
\[
 \|u(t)\|_\infty\le C(u_0)|t|^{-\frac{3}{2}},\qquad t\ne0.
\]
The argument of \cite{KowalskiDecay} applies with Theorem~\ref{thm:main}. See also
\cite{FKVZ,FSZ}.
\end{remark}

\begin{remark}\label{rem:torus}
We expect Theorem~\ref{thm:main} to provide the Euclidean input for
large-data global well-posedness of the focusing energy-critical NLS
on $\mathbb T^3$ under suitable ground-state thresholds. See \cite{KwakTorus,YueT4}.
See also \cite{HTTT3,IPT3} for the small-data and defocusing large-data
theories, respectively.
\end{remark}

Both inequalities in \eqref{eq:threshold} are sharp. Finite-variance
data with $E(u_0)<E(W)$ and
$\|\nabla u_0\|_2>\|\nabla W\|_2$ blow up in finite time
\cite{KM}. At energy $E(W)$, the radial solution $W^-$ of
Duyckaerts--Merle \cite{DMerle} has gradient norm below that of $W$,
yet converges to $W$ and fails to scatter forward in time.
A natural next problem is to understand nonradial dynamics at and
slightly above $E(W)$, following the radial threshold classification
of Duyckaerts--Merle \cite{DMerle} and the near-threshold analysis
of Nakanishi--Roy \cite{NR}.

\subsection{Previous results and the low-frequency problem}
\label{subsec:intro-background}
For the defocusing equation, Bourgain \cite{Bourgain} proved global
well-posedness and scattering for radial $H^1$ data in dimensions
three and four, and Grillakis \cite{Grillakis} proved global
existence for smooth radial data in dimension three. Tao
\cite{TaoRadial} extended radial scattering to $\dot H^1$ data in
all dimensions $d\ge3$. These arguments use spatially localized
Morawetz estimates, with radial symmetry fixing concentration at
the origin. For general data, Colliander, Keel, Staffilani, Takaoka
and Tao \cite{CKSTT} proved the three-dimensional theorem using
induction on energy and frequency-localized interaction Morawetz
estimates. The difference-variable kernel accommodates a moving
concentration center. Ryckman--Visan \cite{RV} proved the
four-dimensional result, and Visan \cite{Visan} treated $d\ge5$.

In the focusing case, energy conservation and the sharp Sobolev
inequality give energy trapping under \eqref{eq:threshold}: the
gradient remains uniformly below that of $W$, and
\[
 \|\nabla u(t)\|_2^2-\|u(t)\|_6^6
 \ge c(u_0)\|\nabla u(t)\|_2^2.
\]
Kenig--Merle \cite{KM} combined this coercivity with Keraani's
Schr\"odinger profile decomposition \cite{Keraani} (see also
\cite{BG,Gerard,MV}) to prove scattering for radial data in dimensions
three, four, and five. Their concentration--compactness
argument reduces a failure of scattering to a critical element:
a minimal-energy nonscattering solution whose orbit is precompact
in $\dot H^1$ modulo translations and rescalings. Radial symmetry
fixes its spatial center, and a localized virial argument excludes
it without finite mass; see also \cite{KenigLectures}.

Without radial symmetry, the virial argument requires control of
the concentration center. For the energy-critical wave equation,
finite momentum and Lorentz invariance provide this control in the
Kenig--Merle rigidity argument \cite{KMWave}. For NLS, finite mass
makes Galilean transformations available. In the kinetic-energy-minimal
formulation of Killip--Visan \cite{KVfocus}, these transformations
force the critical element to have zero momentum. They obtained finite
mass in dimensions $d\ge5$ by proving negative regularity through
additional spatial decay and double Duhamel. In the soliton-like
case, this also gives compactness in $L^2$ modulo translations;
zero momentum then yields sublinear motion of the center and
permits the localized virial argument \cite[Section~8]{KVfocus}. Since the uniform $\dot H^1$ bound controls
high-frequency mass, the missing estimate lies at low frequencies.

The dimensional restriction appears in the double time kernel.
Pairing the forward and backward Duhamel formulas gives a
propagator with time separation $s+\tau$. At unit frequency, the
$L^1$--$L^\infty$ dispersive estimate leads to the model integral
\begin{equation}\label{eq:dim}
 \int_0^T\!\!\int_0^T(1+s+\tau)^{-\frac{d}{2}}\dd s\dd\tau
 \sim
 \begin{cases}
  1,&d\ge5,\\
  \log T,&d=4,\\
  T^{\frac{1}{2}},&d=3,
 \end{cases}
 \qquad T\ge2.
\end{equation}
Dodson \cite{Dodson4} overcame the logarithmic loss in dimension
four by combining long-time Strichartz, localized double Duhamel,
and interaction Morawetz estimates. In dimension three, the power
loss prevents this absolute-value estimate from giving a uniform
mass bound. Compactness alone cannot supply it: the stationary
solution $W$ is almost periodic and has infinite mass.

Earlier nonradial three-dimensional results include an $H^1$
theorem under a localized mass-flux sign condition \cite{Han} and
a rigidity theorem for finite-mass almost-periodic
solutions with frequency scale bounded above and below \cite{CH}.
Related focusing interaction Morawetz arguments at other critical
regularities appear in \cite{DodsonFocus,DM}. Here the strict
ground-state gap gives the nonlinear spacetime control needed for
the general critical element, without an a priori mass bound or an
upper bound on its frequency scale. Together with
\cite{Dodson4,KVfocus}, Theorem~\ref{thm:main} completes the
below-ground-state scattering theory for the focusing energy-critical
NLS in every dimension $d\ge3$.

\subsection{Outline of the proof}\label{subsec:intro-proof}
\paragraph{\textbf{\upshape The remaining critical element.}}
Concentration compactness and the finite-time blowup exclusion of
\cite{KVfocus} reduce the proof to a global almost-periodic solution
whose frequency scale satisfies $N(t)\ge1$ for $t\ge0$. The rapid-cascade exclusion of
Killip--Visan \cite[Section~5]{KVquintic} does not use the defocusing
sign beyond the uniform $\dot H^1$ bound, so it applies to the
critical element obtained here; see Proposition~\ref{prop:reduction}.
Thus the remaining solution satisfies the quasi-soliton condition
\[
 \int_0^\infty N(t)^{-1}\dd t=\infty.
\]
For a sufficiently small fixed $\nu>0$, let $h=P_{>\nu}u$.
Almost periodicity gives
$\|h(t)\|_4^4\gtrsim_u N(t)^{-1}$, so it suffices to prove
\[
 \int_0^\infty\|h(t)\|_4^4\dd t\lesssim_u\nu^{-3}.
\]
This is the frequency-localized interaction Morawetz estimate in
Proposition~\ref{prop:morawetz}. Its defocusing counterpart in
\cite[Theorem~6.1]{KVquintic} is the model for our argument.

\paragraph{\textbf{\upshape Coercivity of the interaction form.}}
The interaction action \cite{CKSTT04} uses a kernel $k(x-y)$ and
avoids differentiating the concentration center. Applied to $h$,
it is finite even when $u$ has infinite mass:
$\|h(t)\|_2\lesssim_u\nu^{-1}$ and
$|M_k(h)|\lesssim_u\nu^{-3}$ whenever $\|\nabla k\|_\infty\le1$.
For $k(z)=|z|$, the identity $-\Delta^2|z|=8\pi\delta_0$
produces the desired quartic term. In the focusing equation,
however, the sixth-power term has the wrong sign.
The bilinear decomposition of Planchon--Vega
\cite[Theorem~2.3, (2.13)]{PV}, combined with the diamagnetic inequality,
reduces the required comparison to one involving $|h|$. The difficulty is that the kinetic term is
weighted by the anisotropic matrix $D^2k*|h|^2$, whereas energy
trapping compares unweighted norms.

A directional dilation followed by the sharp Sobolev inequality
gives a positive gap in every direction:
\[
 \|\partial_\omega|h|\|_2^2-\frac{1}{3}\|h\|_6^6
 \ge\sigma_u>0,\qquad |\omega|=1.
\]
The threshold gap and compactness make this bound uniform in time
and direction for sufficiently small $\nu$. It yields coercivity
for a constant positive semidefinite matrix, with a margin
proportional to its trace. To pass to a varying matrix, we use
\[
 k_R(r)=r-R+Re^{-r/R},\qquad
 |\nabla k_R|\le1,\qquad |D^3k_R|\lesssim R^{-1}\Delta k_R.
\]
For $RN(t)$ large, the last bound allows us to freeze
$D^2k_R*|h|^2$ on the concentration region. The error and the
coercivity margin are both measured against the trace. Almost
periodicity controls the exterior gradient, and a weighted Sobolev
inequality absorbs the exterior potential with constants
independent of $\|h\|_2$. The principal interaction form is
therefore bounded below by the density term
\[
 Q_R(h)\coloneqq-\iint\Delta^2k_R(x-y)|h(x)|^2|h(y)|^2\dd x\dd y;
\]
see Proposition~\ref{prop:static}.

\paragraph{\textbf{\upshape The radius and the spacetime estimate.}}
When $N(t)\sim1$, a fixed large radius gives a positive lower bound
for $Q_R(h(t))$ and excludes the solution. The passage from fixed to
variable scales also appears in Dodson's argument
\cite[Theorem~5.2 and Section~6]{Dodson4} (see also \cite{DodsonFocus}).
For a general quasi-soliton, however, a fixed radius need not retain
a density term comparable to $\|h(t)\|_4^4$. Moreover, $Q_R(h)\to0$ as $R\to\infty$ for fixed
$h$. The radius must therefore be large enough for coercivity
without losing the quartic term.

We determine $R(t)$ implicitly by
\[
 \iint\bigl(\partial_Rk_R(|x-y|)+1\bigr)
       |h(t,x)|^2|h(t,y)|^2\dd x\dd y
 =\frac{32\pi}{\Lambda}R^4.
\]
The parameter $\Lambda$ is chosen sufficiently large, independently
of time and of the frequency cutoff. The defining relation gives
$R(t)\sim_u\|h(t)\|_4^4$, with a lower comparison constant that
increases with $\Lambda$. In view of concentration, this makes
$R(t)N(t)$ large enough for coercivity; the same relation makes
$Q_{R(t)}(h(t))$ control $\|h(t)\|_4^4$. The profile
$\partial_Rk_R+1$ also appears in the radius derivative of the
action. The continuity equation converts the moving-radius
contribution into a nonnegative square, up to a small error from
the frequency truncation. Thus no separate quantitative bound on
$|R'|$ is needed to control its negative part.

The truncation errors require an auxiliary spacetime estimate for
the nonlinearity. A frequency recurrence first gives
$u\in L^\infty_tL^4_x$ and hence a uniform bound for
$\|u(t)\|_5^5$. Cancellations in the truncated nonlinearity expose
low-frequency factors, which the no-waste Duhamel formula bounds
by future averages of $\|u(t)\|_5^5$. The remaining source terms
are controlled by long-time Strichartz estimates
\cite{KVquintic}, building on \cite{DodsonMass}. The concentration inequality bounds the scale
integral in these estimates by the quartic integral. Almost
periodicity also makes the low-frequency gradient and
$\nu\|h\|_{L^\infty_tL^2_x}$ small as $\nu\to0$, providing the
small coefficients needed for absorption.
Since the future averages can extend beyond the
interval under consideration, the forcing and Morawetz inequalities
are absorbed together on finite intervals, with the cutoff chosen
independently of the endpoint. Boundedness of $\|u(t)\|_5^5$
allows this absorption before its global time integrability is
known. This proves the global quartic
bound and excludes the quasi-soliton.

Our work on the three-dimensional defocusing cubic NLS
\cite{SuZhaoCubic} also uses a directional interaction identity,
there to obtain long-time bilinear estimates for the $I$-method.
Here the ground-state gap and the density-dependent radius yield
the required quartic spacetime bound.

\subsection{Organization and notation}\label{subsec:intro-notation}
Section~\ref{sec:reduction} gives the reduction, states the required
spacetime estimate, and records the interaction identity.
Section~\ref{sec:lp} proves the spatial and low-frequency bounds.
Section~\ref{sec:static} establishes coercivity and treats the
soliton-like case; Section~\ref{sec:radius} chooses the radius and
controls its motion. Sections~\ref{sec:source} and
\ref{sec:interaction} estimate the truncation errors and complete
the proof.

Unless specified otherwise, spatial norms are over $\R^3$ and
$\|\cdot\|_p$ denotes the $L^p_x$ norm. We write $X\lesssim Y$ if $X\le CY$ with $C$
independent of the parameters under consideration; subscripts mark
allowed dependencies. The notation $X\sim Y$ means both
$X\lesssim Y$ and $Y\lesssim X$. Our Fourier transform is
\[
 \widehat f(\xi)\coloneqq \int_{\R^3}e^{-2\pi\ii x\cdot\xi}f(x)\dd x.
\]
Thus $\widehat{\partial_j f}(\xi)=2\pi\ii\xi_j\widehat f(\xi)$,
and convolution with $K$ has multiplier $\widehat K$.
For spatial functions, $b_{ij}\coloneqq\partial_i\partial_jb$;
$D^m$ denotes derivatives of order $m$, and $Dq$ the Jacobian of $q$.
For real symmetric matrices, positivity and lower bounds are understood
in the sense of quadratic forms; $\|\cdot\|_{\rm op}$ denotes the
Euclidean operator norm. For $\omega\in S^2$, set
$\partial_\omega=\omega\cdot\nabla$.

Fix a real radial $\varphi\in C_c^\infty(\R^3)$ satisfying
$0\le\varphi\le1$, $\varphi=1$ on $B(0,1)$, and
$\supp\varphi\subset B(0,2)$. The associated smooth
Littlewood--Paley projections are
\[
 \widehat{P_{\le L}f}(\xi)\coloneqq \varphi(\frac{\xi}{L})\widehat f(\xi),\qquad
 P_{>L}\coloneqq 1-P_{\le L},\qquad P_L\coloneqq P_{\le L}-P_{\le \frac{L}{2}}.
\]
All annular sums and suprema range over dyadic scales; fixed
enlargements of annuli are allowed. We sum over repeated spatial
indices and omit time intervals in spacetime norms when clear.

\section{Reduction to the quasi-soliton scenario}\label{sec:reduction}
We record the reduction to the quasi-soliton scenario using \cite{KM,KVfocus,KVquintic}, state the high-frequency spacetime estimate that excludes it, and fix the interaction identity used to prove that estimate.

\subsection{The critical element and its frequency truncations}
For a solution $u$ with initial data satisfying \eqref{eq:threshold},
the energy-trapping estimates
\cite[Theorem 3.9 and Corollary 3.13]{KM} furnish constants $B,\gamma>0$ such that,
throughout the maximal lifespan,
\begin{equation}\label{eq:trapping}
 \begin{gathered}
 \|\nabla u(t)\|_2^2\le B^2\le(1-\gamma)G<G,\\
 \|\nabla u(t)\|_2^2-\|u(t)\|_6^6
 \ge(2\gamma-\gamma^2)\|\nabla u(t)\|_2^2,\qquad
 E(u(t))\ge \frac{\|\nabla u(t)\|_2^2}{3}.
 \end{gathered}
\end{equation}
We use the local theory and stability estimates of \cite[Section~2]{KM}.
For a solution $v$ of an inhomogeneous Schr\"odinger equation on a
compact interval $I$, the Strichartz estimates \cite{KT} give
\[
 \|v\|_{L^q_tL^r_x(I)}
 \lesssim \|v(t_0)\|_2
 +\|(\ii\partial_t+\Delta)v\|_{L^2_tL^{\frac{6}{5}}_x(I)},
 \qquad t_0\in I,
\]
where $(q,r)$ is Schr\"odinger-admissible:
$2\le q,r\le\infty$ and $\frac{2}{q}+\frac{3}{r}=\frac{3}{2}$.
The implicit constant is independent of $I$.
In particular, the local theory gives
$\nabla u\in L^q_tL^r_x(I)$ for every such pair and every compact
subinterval $I\subset I_{\max}$.

If Theorem~\ref{thm:main} fails, the local theory, stability, and
energy trapping allow the concentration--compactness arguments of
\cite{KM,KVfocus} to produce a counterexample of minimal energy. Its infinite scattering
size need not correspond to a finite lifespan. Its orbit is
precompact modulo scaling and translation, as follows.

\begin{definition}\label{def:ap}
A nonzero solution is \emph{almost periodic modulo symmetries} if there
exist functions $N:I_{\max}\to(0,\infty)$ and $x:I_{\max}\to\R^3$ such
that the normalized orbit
\begin{equation}\label{eq:compact}
 \Kcal\coloneqq \{N(t)^{-\frac{1}{2}}u(t,x(t)+\frac{\cdot}{N(t)}):
                         t\in I_{\max}\}
\end{equation}
is precompact in $\dot H^1$ and satisfies
$\inf_{v\in\Kcal}\|\nabla v\|_2>0$.
\end{definition}
The lower bound excludes zero from the compact closure of the normalized orbit. The next proposition records the normalization and the remaining scale behavior, together with the no-waste Duhamel formula. Write $F(v)\coloneqq |v|^4v$.

\begin{proposition}\label{prop:reduction}
If Theorem~\ref{thm:main} fails, there is a nonzero almost-periodic
solution of \eqref{eq:nls} with maximal lifespan $I_{\max}=\R$ and
energy $E(u)=E_c\in(0,E(W))$, satisfying
\[
 \sup_t\|\nabla u(t)\|_2<\|\nabla W\|_2,\qquad
 \|u\|_{L^{10}_{t,x}([0,\infty)\times\R^3)}=\infty.
\]
Moreover, the forward half-line has a decomposition
$[0,\infty)=\bigcup_kJ_k$ into characteristic intervals (in the sense
of \cite{KVquintic}) such that
\[
 N|_{J_k}=N_k\ge1,\qquad |J_k|\sim_uN_k^{-2}.
\]
These intervals satisfy
$\sup_k\|\nabla u\|_{L^2_tL^6_x(J_k)}\lesssim_u1$, and
\begin{equation}\label{eq:quasi-soliton}
 \int_0^\infty N(t)^{-1}\dd t=\infty.
\end{equation}
The solution also obeys the no-waste Duhamel formula
\begin{equation}\label{eq:no-waste}
 u(t)=-\ii\lim_{T\to\infty}\int_t^T
             e^{\ii(t-s)\Delta}F(u(s))\dd s,
\end{equation}
for every $t\ge0$, with the limit taken weakly in $\dot H^1$.
\end{proposition}

\begin{proof}
Let $E_c\in(0,E(W))$ be the infimum of the energies of
solutions satisfying \eqref{eq:threshold} and failing to scatter.
The nonradial concentration--compactness argument of
\cite[Propositions~4.1--4.2]{KM}, using \eqref{eq:trapping}, yields
a critical element at energy $E_c$. Centering the approximating
sequence so that both one-sided scattering sizes diverge and applying
the same profile extraction gives blowup in both directions and
almost periodicity on the full lifespan; cf.\
\cite[proof of Proposition~3.1 and Section~3.2]{KVfocus}.

The normalization construction of \cite[Section~4]{KVfocus} applies
within this energy-minimal class: symmetries and normalized strong
$\dot H^1$ limits preserve $E_c$ and the trapped gradient gap, and
the common compactness modulus passes to the normalized limits.
If such a limit scattered in either time direction, long-time
stability would force the approximating solutions to scatter in
that direction, a contradiction. The construction yields either
a finite lifespan endpoint or a global
solution with $\inf_tN(t)>0$. The former is excluded by
\cite[Theorem~5.1]{KVfocus}, which holds for $d\ge3$ under the
almost-periodicity, blowup, and gradient-gap hypotheses just obtained.

After a fixed rescaling, the sign-independent local theory in
\cite[Lemmas~1.5 and 1.7 and the ensuing discussion]{KVquintic}
gives the stated characteristic intervals and local bound, with
$N_k\ge1$. Lemma~1.18 of~\cite{KVfocus} gives
\eqref{eq:no-waste}. Finally, the rapid-cascade exclusion in
\cite[Section~5, Lemma~5.1 and Theorem~5.2]{KVquintic} applies:
its only use of the defocusing sign is the uniform $\dot H^1$ bound,
supplied here by \eqref{eq:trapping}. If the scale integral were
finite, the long-time Strichartz and no-waste estimates would give
finite mass, and the same argument would force this conserved mass
to vanish. This contradicts $E_c>0$ and proves
\eqref{eq:quasi-soliton}.
\end{proof}

Following \cite{KVquintic}, we call \eqref{eq:quasi-soliton} the \emph{quasi-soliton scenario}. It includes the \emph{soliton-like} case $N(t)\sim1$, but also permits unbounded frequency scales. In particular, the reduction does not allow us to assume a fixed scale.

Fix the solution $u$ supplied by Proposition~\ref{prop:reduction}, and set $B\coloneqq\sup_{t\ge0}\|\nabla u(t)\|_2$, so that $B<\|\nabla W\|_2$. Unless stated otherwise, all time norms and suprema are taken over $[0,\infty)$. Constants may depend on the compact closure of the normalized orbit and on the threshold gaps.

Fix $0<\nu<1$ and write
\begin{equation}\label{eq:fixed-notation}
 h\coloneqq P_{>\nu}u,\qquad l\coloneqq P_{\le\nu}u,\qquad a(t)\coloneqq \|h(t)\|_4^4.
\end{equation}
The high-frequency part belongs to $C_tH^1_x$ and satisfies
$\|h(t)\|_2\lesssim_u\nu^{-1}$, although $u$ need not have finite mass.
The dependence of $h$, $l$, and $a$ on $\nu$ is suppressed. Set
\[
 M_\nu\coloneqq\sup_t\|h(t)\|_2\le\frac{B}{\nu},\qquad
 d_\nu\coloneqq\nu M_\nu,\qquad
 \eta_\nu\coloneqq\sup_t\|\nabla l(t)\|_2.
\]
For $z\in\{u,h\}$, define the normalized profiles
\[
 \widetilde z(t,y)\coloneqq N(t)^{-\frac{1}{2}}
 z\bigl(t,x(t)+\frac{y}{N(t)}\bigr).
\]

Almost periodicity gives the following frequency smallness and local concentration; see \cite[Remark~1.3, (6-13), and the proof of Theorem~7.1]{KVquintic}. We record the uniform estimates and check that the focusing thresholds survive truncation.

\begin{proposition}
\label{prop:frequency-properties}
For the solution fixed above,
\begin{equation}\label{eq:dsmall}
 \eta_\nu\longrightarrow0,\qquad d_\nu\longrightarrow0
 \qquad\text{as }\nu\to0,
\end{equation}
and
\[
 \sup_t\|\widetilde h(t)-\widetilde u(t)\|_{\dot H^1}\le\eta_\nu.
\]
There exist $\nu_0,R_0,c_{\Kcal}>0$, independent of time, such that
for $0<\nu\le\nu_0$,
\[
 \sup_t E(h(t))<E(W),\qquad
 \sup_t\|\nabla h(t)\|_2<\|\nabla W\|_2,
\]
and, for every $t\ge0$,
\begin{equation}\label{eq:a-bounds}
 c_{\Kcal}N(t)^{-1}
 \le\int_{|x-x(t)|\le R_0/N(t)}|h(t,x)|^4\dd x
 \le a(t).
\end{equation}
The constants depend only on the compact normalized orbit and its
strict threshold gaps.
\end{proposition}

\begin{proof}
Define
\[
 \omega(c)\coloneqq4\pi^2\sup_t
 \int_{|\xi|\le2c}|\xi|^2|\widehat u(t,\xi)|^2\dd\xi.
\]
Normalization~\eqref{eq:compact} changes the cutoff to $2c/N(t)\le2c$, so compactness gives $\omega(c)\to0$ as $c\to0$. For $0<\nu<c$, the multiplier bounds give
\[
 \eta_\nu^2\le\omega(\nu),\qquad
 d_\nu^2\le\omega(c)+\nu^2c^{-2}B^2.
\]
First let $\nu\to0$ and then $c\to0$ to obtain \eqref{eq:dsmall}.
Scaling invariance gives
\[
 \|\widetilde h(t)-\widetilde u(t)\|_{\dot H^1}
 =\|h(t)-u(t)\|_{\dot H^1}\le\eta_\nu.
\]
Moreover, the multiplier bounds and Sobolev give
\[
 \|\nabla h(t)\|_2^2\le(1-\gamma)G,\qquad
 |E(h(t))-E(u)|\lesssim(B+B^5)\eta_\nu.
\]
Thus the two strict threshold inequalities persist for all
sufficiently small $\nu$, uniformly in time.

The compact closure of $\Kcal$ excludes zero. Continuity of the local $L^4$ norm in $\dot H^1$ and a finite cover therefore give $R_0,b_0>0$ such that
\[
 \inf_{v\in\Kcal}\|v\|_{L^4(B(0,R_0))}\ge b_0.
\]
After decreasing $\nu_0$, Sobolev and H\"older give
\[
 \|\widetilde h-\widetilde u\|_{L^4(B(0,R_0))}
 \lesssim R_0^{\frac{1}{4}}\eta_\nu\le\frac{b_0}{2}.
\]
Consequently,
\[
 \int_{|x-x(t)|\le R_0/N(t)}|h(t,x)|^4\dd x
 \ge\frac{b_0^4}{16N(t)}.
\]
Taking $c_{\Kcal}=\frac{b_0^4}{16}$ proves \eqref{eq:a-bounds}.
\end{proof}

The smallness in~\eqref{eq:dsmall} will allow us to absorb the
frequency-truncation errors and the loss from the moving radius.
The normalized-profile bound permits application of
Proposition~\ref{prop:static} to $\widetilde h(t)$ at radius
$N(t)R(t)$, once this radius is sufficiently large. The
concentration estimate~\eqref{eq:a-bounds}, together with
Lemma~\ref{lem:radius}, will ensure that condition. It also gives
\[
 \int_I N(t)^{-1}\dd t\lesssim_u\int_Ia(t)\dd t
\]
on every compact $I\subset[0,\infty)$, controlling the scale
integral in the long-time Strichartz estimate.

\subsection{The high-frequency spacetime estimate}
The lower bound \eqref{eq:a-bounds} reduces the remaining rigidity argument to a global $L^4_{t,x}$ estimate for $h$. The frequency-localized interaction Morawetz estimate of \cite[Theorem~6.1]{KVquintic} is the model for this bound. We prove it together with the nonlinear estimate needed to control the frequency-truncation errors.
\begin{proposition}[Frequency-localized interaction Morawetz estimate]
\label{prop:morawetz}
Let $u$ be as in Proposition~\ref{prop:reduction}. There are constants
$C_u,\nu_0>0$ such that, for $0<\nu\le\nu_0$,
\begin{equation}\label{eq:morawetz-global}
 \int_0^\infty\!\int_{\R^3}|P_{>\nu}u(t,x)|^4\dd x\dd t
 +\int_0^\infty\|F(u(t))\|_1^2\dd t\le C_u\nu^{-3}.
\end{equation}
\end{proposition}

The first term in \eqref{eq:morawetz-global}, together with \eqref{eq:a-bounds}, contradicts \eqref{eq:quasi-soliton}. Thus this proposition suffices to prove Theorem~\ref{thm:main}. The second term is an auxiliary estimate: it closes the source bounds needed to obtain the first. We establish both estimates simultaneously in Section~\ref{sec:interaction}.

\subsection{The interaction identity}\label{subsec:interaction-identity}
We record the frequency-localized interaction identity
\cite{CKSTT,KVquintic,PV}, separating the principal term from the
truncation errors. We begin with a time-independent real even
weight $k$.

The high-frequency component $h=P_{>\nu}u$ satisfies
\[
 (\ii\partial_t+\Delta)h=-F(h)-\mathcal E,\qquad
 \mathcal E\coloneqq P_{>\nu}F(u)-F(h).
\]
The term $\mathcal E$ measures the failure of frequency projection to commute with the nonlinearity. Write $\rho\coloneqq |h|^2$ and $p\coloneqq \Imn(\bar h\nabla h)$ for the mass and momentum densities. The interaction Morawetz action is
\begin{equation}\label{eq:action}
 M_k(h)\coloneqq 2\iint_{\R^3\times\R^3}\rho(y)\nabla k(x-y)\cdot p(x)\dd x\dd y.
\end{equation}
For bounded $\nabla k$,
\[
 |M_k(h)|\le2\|\nabla k\|_\infty\|h\|_2^3\|\nabla h\|_2
 \lesssim_{u,k}\nu^{-3}.
\]
To differentiate this action, use the mass and momentum source brackets
\[
 \{z,h\}_m\coloneqq \Imn(z\bar h),\qquad
 \{z,h\}_p\coloneqq \Ren(z\nabla\bar h-h\nabla\bar z).
\]
The forced interaction identity
\cite[Proposition~6.2]{KVquintic}, with the focusing sign, gives
\[
 \frac{\mathrm d}{\mathrm dt}M_k(h(t))
 =\BB_k(h(t))+\mathrm{Err}_{k,\nu}(t).
\]
The principal term $\BB_k(h)$ contains no $\mathcal E$. Write
\[
 Q_k(h)\coloneqq -\iint_{\R^3\times\R^3}\Delta^2k(x-y)\rho(x)\rho(y)\dd x\dd y
\]
for its density contribution. Then
\begin{align}
 \BB_k(h)\coloneqq {}&Q_k(h)
 +4\iint_{\R^3\times\R^3} k_{jk}(x-y)
 \big[\rho(y)\Ren(\partial_j\bar h\,\partial_kh)(x)-p_j(y)p_k(x)\big]
                                                    \dd x\dd y\notag\\
 &-\frac{4}{3}\iint_{\R^3\times\R^3}\Delta k(x-y)|h(x)|^6\rho(y)\dd x\dd y.
 \label{eq:principal}
\end{align}
The source term is expressed through convolutions with the mass and momentum densities. Its estimates in Section~\ref{sec:source} use the spatial bounds
\begin{equation}\label{eq:kernel-source-bounds}
 |\nabla k(x)|+|x||D^2k(x)|+|x|^2|D^3k(x)|\le C_k.
\end{equation}
When $k$ depends on time, these spatial bounds hold uniformly in $t$.
The derivatives are understood in the weak sense, and the bounds hold
almost everywhere. Define
the vector field $q$ and scalar function $q_m$ by
\[
 q(x)\coloneqq \int_{\R^3}\nabla k(x-y)\rho(y)\dd y,\qquad
 q_m(y)\coloneqq \int_{\R^3}\nabla k(x-y)\cdot p(x)\dd x.
\]
Suppressing the time argument in the spatial functions, the source
contribution is
\begin{equation}\label{eq:source-identity}
 \mathrm{Err}_{k,\nu}(t)
 \coloneqq -2\int_{\R^3}q\cdot\{\mathcal E,h\}_p\dd x
           -4\int_{\R^3}q_m\{\mathcal E,h\}_m\dd x.
\end{equation}
We will use a family $k_R$ with a solution-dependent radius $R=R(t)$. For a positive locally Lipschitz radius, set $k(t)=k_{R(t)}$ and $M(t)=M_{k(t)}(h(t))$. The chain rule then adds the radius-variation term:
\begin{equation}\label{eq:full-identity}
 M'=\BB_{k(t)}(h)+\mathrm{Err}_{k(t),\nu}+D,
 \qquad
 D(t)=R'(t)\left.\partial_R M_{k_R}(h(t))\right|_{R=R(t)}.
\end{equation}

For the kernels $k_R$ constructed in Section~\ref{sec:static},
\eqref{eq:full-identity} holds at energy regularity in the integrated
sense. To see this, fix a compact interval $I$ and a positive locally
Lipschitz radius $R$. Apply the forced interaction identity
\cite[Proposition~6.2]{KVquintic} after spatial regularization of the
projected equation and kernel, keeping the original function $R(t)$
unchanged throughout. Local theory gives
\[
 h\in C(I;H^1)\cap L^2(I;W^{1,6}),\qquad
 \mathcal E\in L^2(I;W^{1,\frac{6}{5}}).
\]
For a spatial mollifier $J_\epsilon$, the function
$h_\epsilon=J_\epsilon h$ solves the regularized equation with source
\[
 \mathcal E_\epsilon
 =J_\epsilon\mathcal E+J_\epsilon F(h)-F(h_\epsilon)
 \longrightarrow\mathcal E
 \quad\text{in }L^2(I;W^{1,\frac{6}{5}}).
\]
The convergence follows from the displayed bounds and the quintic
product estimate; also $h_\epsilon\to h$ in
$C(I;H^1)\cap L^2(I;W^{1,6})$. Consequently, the mass densities converge in
$C(I;L^1\cap L^2)$, the momentum densities converge in $C(I;L^1)$,
and H\"older's inequality gives convergence of the source terms in
$L^1(I)$. Since $R(I)$ is a compact subset of $(0,\infty)$,
$D^2k_R$ and $\partial_R\nabla k_R$ are uniformly bounded there.
These bounds handle the remaining products and, since
$R'\in L^\infty(I)$, the radius-variation term. The density term
passes to the limit through the bounded Fourier multiplier
\eqref{eq:SQ}; thus the weak fourth derivative of the kernel causes
no additional term. Removing the spatial cutoffs and mollifiers
therefore gives local absolute continuity of $M$ and
\begin{equation}\label{eq:integrated-identity}
 M(t)-M(s)=\int_s^t
 \big[\BB_{k(\tau)}(h(\tau))
 +\mathrm{Err}_{k_{R(\tau)},\nu}(\tau)+D(\tau)\big]\dd\tau.
\end{equation}
This applies to constant radii and, by
Lemma~\ref{lem:radius-variation}, to the radius chosen in
Section~\ref{sec:radius}.

For $k(x)=|x|$, the density term is exactly $Q_k(h)=8\pi a(t)$.
The negative potential term in \eqref{eq:principal} is the obstruction.
Writing it in convolution form gives
\[
 \iint_{\R^3\times\R^3}\Delta k(x-y)|h(x)|^6\rho(y)\dd x\dd y
 =\int_{\R^3}(\Delta k*|h|^2)(x)|h(x)|^6\dd x.
\]
The gradient-square term is weighted by $D^2k*|h|^2$, whereas the potential term is weighted by its trace, $\Delta k*|h|^2$. This matrix can vary with $x$ and weight different directions unequally, so the unweighted trapping estimate \eqref{eq:trapping} does not directly compare the two terms. For the distance weight,
\[
 D^2|z|=\frac{1}{|z|}
 \left(I-\frac{z\otimes z}{|z|^2}\right),\qquad z\ne0.
\]
The zero radial eigenvalue of $D^2|z|$ illustrates the limitation of convexity alone. Introducing a radius $R$ allows the convolved Hessian to vary little relative to its trace on the concentration region when $RN(t)$ is large. Section~\ref{sec:static} constructs the weight and proves its spatial coercivity. Section~\ref{sec:radius} then chooses the radius $R(t)$ so that
\begin{equation}\label{eq:spatial-target}
 \BB_{k(t)}(h(t))\ge c_u a(t),\qquad
 D(t)\ge-C_u(\eta_\nu^2+\nu)a(t).
\end{equation}
Integrating \eqref{eq:full-identity} on $[0,T]$ and using the
bound on the action gives
\[
 (c_u-C_u(\eta_\nu^2+\nu))\int_0^T a(t)\dd t
 \le C_u\nu^{-3}
       +\left|\int_0^T\mathrm{Err}_{k(t),\nu}(t)\dd t\right|.
\]
The source integral must be controlled before either spacetime bound in \eqref{eq:morawetz-global} is available. The required low-frequency estimates also control the mass source in the continuity equation used to estimate $D(t)$. We establish these estimates before constructing the weight.

\section{Low-frequency bounds}\label{sec:lp}
The Morawetz argument requires control of the low-frequency terms introduced by truncation. We first establish a uniform spatial $L^4$ bound for the critical element, then use the no-waste formula to estimate these terms. Such integrability does not follow from energy regularity and almost periodicity alone. The decay of the ground state suggests seeking an $L^p_x$ bound with $p>3$:
\[
 W(x)\sim |x|^{-1}\quad (|x|\to\infty),\qquad
 W\notin L^2(\R^3),\qquad W\in L^p(\R^3)\quad (p>3).
\]
We choose $p=4$ because H\"older's inequality gives
\[
 \|F(u(t))\|_1^2=\|u(t)\|_5^{10}
 \le\|u(t)\|_4^4\|u(t)\|_6^6\lesssim_B\|u(t)\|_4^4.
\]
Thus a uniform spatial $L^4$ bound gives an $L^1_x$ bound for $F(u)$ and connects the forcing estimate to the quartic Morawetz estimate.

The same bound helps estimate the source in \eqref{eq:source-identity}. Write $\mathcal E=G_1-g$, where $G_1=F(u)-F(h)$ and $g=P_{\le\nu}F(u)$. The mass bracket satisfies
\[
 \Imn(G_1\bar h)=-\Imn(F(u)\bar l),
\]
and integration by parts exposes $F(u)\nabla\bar l$ in the momentum term; see \eqref{eq:Gcancel}. Thus the source estimate calls for $L^\infty_x$ bounds on $l$ and $\nabla l$, which we obtain below from the no-waste formula.
\subsection{Spatial integrability of the solution and nonlinearity}
For dyadic frequencies $L\in2^{\mathbb Z}$, the frequency-localized
dispersive estimate gives
\[
 \int_0^\infty\|e^{-\ii s\Delta}P_L\|_{L^1_x\to L^\infty_x}\dd s
 \lesssim\int_0^\infty\min\{L^3,s^{-\frac{3}{2}}\}\dd s=3L.
\]
The factor of $L$ suggests estimating the seminorm $\sup_LL^{-1}\|P_Lu\|_\infty$ through the no-waste formula. Together with the gradient bound, this seminorm controls the spatial $L^4$ norm by the refined Sobolev inequality \cite[Th\'eor\`eme 2]{GMO}:
\begin{equation}\label{eq:refined}
 \|v\|_4^4\lesssim
 \left(\sup_LL^{-1}\|P_Lv\|_\infty\right)^2\|\nabla v\|_2^2.
\end{equation}
Here $v\in\dot H^1$ and the seminorm on the right is finite. This is
the cited theorem with $n=3$, $p=2$, $q=4$, and $\alpha=s=1$,
expressed using the displayed Littlewood--Paley seminorm; no a priori
$L^4$ integrability is required.
The next lemma bounds the seminorm uniformly in time, improving Bernstein's estimate $\|P_Lu\|_{L^\infty_{t,x}}\lesssim BL^{\frac{1}{2}}$ to an estimate of order $L$ at low frequencies. This step does not use the strict ground-state gap. The recurrence is analogous to the additional-decay argument of \cite[Section~6]{KVfocus}; see also \cite[Theorem~4.3]{Dodson4} for the four-dimensional endpoint estimate.

\begin{lemma}\label{lem:l4}
Let $u$ be a forward-global almost-periodic strong solution of
\eqref{eq:nls} on $[0,\infty)$ with $N(t)\ge1$ and
$\sup_{t\ge0}\|\nabla u(t)\|_2\le B<\infty$, satisfying
the forward no-waste formula \eqref{eq:no-waste}. Then
\begin{equation}\label{eq:besov}
 A_*\coloneqq \sup_{L\in2^{\mathbb Z}}L^{-1}\|P_Lu\|_{L^\infty_{t,x}}<\infty.
\end{equation}
In particular, $U\coloneqq \sup_{t\ge0}\|u(t)\|_4<\infty$ and
$\sup_{t\ge0}\|F(u(t))\|_1<\infty$.
\end{lemma}
\begin{proof}
Set $a_L\coloneqq L^{-1}\|P_Lu\|_{L^\infty_{t,x}}$. Bernstein gives $a_L\lesssim BL^{-\frac{1}{2}}$, so only low frequencies need a new bound. Fix $\eta>0$, to be chosen after deriving the recurrence.
Almost periodicity and $N(t)\ge1$ allow us to choose $L_0>0$ so that
\[
 v\coloneqq P_{\le L_0}u,\qquad\|\nabla v\|_{L^\infty L^2}\le\eta.
\]
On the normalized orbit \eqref{eq:compact}, this projection has
cutoff $L_0/N(t)\le L_0$, so compactness makes the smallness uniform
in $t$.
The high-frequency part $u-v$ belongs uniformly to $L^2_x\cap L^6_x$:
$\|u-v\|_{L^\infty L^2}\lesssim L_0^{-1}B$ and
$\|u-v\|_{L^\infty L^6}\lesssim B$. Interpolation and H\"older's
inequality give
\[
 \|u-v\|_{L^\infty L^3}\le C_{L_0,B},\qquad
 \|F(u)-F(v)\|_{L^\infty L^1}
 \lesssim\|u-v\|_{L^\infty L^3}
       (\|u\|_{L^\infty L^6}+\|v\|_{L^\infty L^6})^4
 \le C_{L_0,B}.
\]

\emph{Step 1: The dyadic recurrence.} The difference $F(u)-F(v)$ supplies the frequency-independent term. To control $P_LF(v)$, separate frequencies well below $L$ so that the purely low-frequency quintic term disappears under $P_L$. The small gradient of $v$ will supply the factor $\eta^4$. Set $r\coloneqq P_{\le \frac{L}{100}}v$ and $z\coloneqq v-r$.
The Fourier supports of $r$ and $F(r)$ lie in
$B(0,\frac{L}{50})$ and $B(0,\frac{L}{10})$, respectively. Hence
$P_LF(r)=0$. Moreover, $\widehat z$ vanishes on
$B(0,\frac{L}{100})$, and the bounds above imply
\[
 \|z\|_2\lesssim L^{-1}\eta,\qquad
 \|\nabla z\|_2+\|\nabla r\|_2+\|z\|_6+\|r\|_6\lesssim\eta.
\]
For each fixed $L$, the initial Bernstein bound gives
\[
 \sup_{M\ge cL}a_M+\sum_{M\le CL}\frac{M}{L}a_M
 \lesssim BL^{-\frac{1}{2}}<\infty.
\]
Thus the dyadic seminorm of $z$ is finite at each fixed $L$, although the global seminorm of $u$ is still unknown. We may therefore apply~\eqref{eq:refined} to $z$ without assuming~\eqref{eq:besov}. Using $P_Mz=(1-P_{\le \frac{L}{100}})P_{\le L_0}P_Mu$ and the uniform $L^1$ kernel bounds gives
\begin{gather*}
 \sup_M M^{-1}\|P_Mz\|_{L^\infty_{t,x}}
 \lesssim\sup_{M\ge cL}a_M<\infty,\\
 \|z\|_4^2\lesssim \|\nabla z\|_2
              \sup_M M^{-1}\|P_Mz\|_\infty
       \lesssim\eta\sup_{M\ge cL}a_M,\\
 \|F(z)\|_1\le\|z\|_4^2\|z\|_6^3
                 \lesssim\eta^4\sup_{M\ge cL}a_M.
\end{gather*}
The multiplier bounds also give
$\|r\|_{L^\infty_{t,x}}\lesssim\sum_{M\le CL}M a_M$.
In the expansion $F(v)=(r+z)^3(\bar r+\bar z)^2$, the term $F(r)$
vanishes under $P_L$. For each mixed term with $k$ factors of $z$ or
$\bar z$, $1\le k\le4$, place one factor of $r$ in $L^\infty$,
one of $z$ in $L^2$, and the remaining three factors in $L^6$.
H\"older's inequality gives
\begin{align*}
 \||z|^k|r|^{5-k}\|_1
 &\le\|r\|_\infty\|z\|_2\|z\|_6^{k-1}\|r\|_6^{4-k}\\
 &\lesssim\eta^4\sum_{M\le CL}\frac{M}{L}a_M.
\end{align*}

Set $\mathcal R_L\coloneqq [F(u)-F(v)]+[F(v)-F(r)]$.
Since $P_LF(r)=0$, we have $P_LF(u)=P_L\mathcal R_L$.
The preceding estimates put $\mathcal R_L$ in
$L^\infty_tL^1_x$ for each fixed $L$.
Applying the forward no-waste formula at frequency $L$ gives
\[
 P_Lu(t)=-\ii\int_t^\infty
 e^{\ii(t-s)\Delta}P_L\mathcal R_L(s)\dd s.
\]
The frequency-localized dispersive estimate above makes the integral absolutely convergent in $L^\infty_x$. Its distributional limit agrees with $P_L$ applied to the weak $\dot H^1$ limit in \eqref{eq:no-waste}. Integrating the dispersive kernel contributes a factor $O(L)$, giving
\[
 \|P_Lu\|_{L^\infty_{t,x}}
 \lesssim L\|\mathcal R_L\|_{L^\infty_tL^1_x}.
\]
Substituting the bounds for $\mathcal R_L$ and dividing by $L$ gives
\begin{equation}\label{eq:recurrence}
 a_L\le C_{L_0,B}+C\eta^4
 \left[\sup_{M\ge cL}a_M+\sum_{M\le CL}(\frac{M}{L})a_M\right].
\end{equation}
The first constant may depend on the low-frequency cutoff chosen from the compact orbit. The coefficient of $\eta^4$ is independent of $L_0$, $L$, and $\eta$, which permits the smallness parameter to be fixed before the cutoff.

\emph{Step 2: Absorb the recurrence.} The supremum over all frequencies in~\eqref{eq:recurrence} is not yet known to be finite. Fix $\alpha=\frac{3}{4}$ and introduce a weighted supremum that is finite by the initial Bernstein bound:
\[
 w_\eps(L)\coloneqq \min\{1,(\frac{L}{\eps})^\alpha\},
 \qquad A_\eps\coloneqq \sup_Lw_\eps(L)a_L\lesssim B\eps^{-\frac{1}{2}}<\infty.
\]
The conditions $\alpha>\frac{1}{2}$ and $\alpha<1$ make
$A_\eps$ finite and the required dyadic sum convergent, respectively.
Indeed, the weight ratios satisfy
\[
 \sup_{M\ge cL}\frac{w_\eps(L)}{w_\eps(M)}\lesssim_c1,\qquad
 \sum_{M\le CL}\frac{M}{L}\frac{w_\eps(L)}{w_\eps(M)}
          \lesssim_C\sum_{j\ge0}2^{-j(1-\alpha)}<\infty,
\]
and hence
\[
 A_\eps\le C_{L_0,B}+C_\alpha\eta^4A_\eps.
\]
Choose $\eta$ so that $C_\alpha\eta^4\le\frac{1}{2}$, then fix
$L_0$ accordingly. We obtain $A_\eps\le2C_{L_0,B}$ uniformly in
$\eps$. At each fixed $L$, let $\eps\to0$; taking the supremum
afterwards proves \eqref{eq:besov}. Applying \eqref{eq:refined} gives
\begin{equation}\label{eq:Fstar}
 U^4\lesssim A_*^2B^2<\infty,
 \qquad\sup_t\|F(u(t))\|_1\le U^2\sup_t\|u(t)\|_6^3
                                      \le CU^2B^3\eqqcolon F_*.
\qedhere
\end{equation}
\end{proof}

For the high-frequency part defined in \eqref{eq:fixed-notation},
the uniform $L^4$ multiplier bound now gives
\[
 a(t)=\|P_{>\nu}u(t)\|_4^4\le CU^4
 \qquad(t\ge0,\ 0<\nu<1),
\]
with a constant independent of $\nu$. Together with
\eqref{eq:a-bounds}, this provides both bounds on $a(t)$ used below.

\subsection{Estimates from the no-waste Duhamel formula}
Set $f(t)\coloneqq\|F(u(t))\|_1=\|u(t)\|_5^5$, so $f\le F_*$ by \eqref{eq:Fstar}. The no-waste formula will estimate $l$ and $\nabla l$ and compare $f^2$ with $a$ through future averages of the forcing. Define the one-sided Hardy--Littlewood maximal function
\[
 F_+(t)=\sup_{s>0}\frac1s\int_t^{t+s}f(\tau)\dd\tau.
\]
We have $F_+\le F_*$. For a compact interval
$I\subset[0,\infty)$, write
\[
 K_a(I)\coloneqq \int_Ia(t)\dd t
       =\|P_{>\nu}u\|_{L^4_{t,x}(I\times\R^3)}^4.
\]

\begin{proposition}\label{prop:low-forcing}
For all sufficiently small $\nu>0$ and almost every $t\ge0$,
\begin{equation}\label{eq:deep-low-H}
 \|l(t)\|_\infty\lesssim\nu F_+(t),\qquad
 \|\nabla l(t)\|_\infty\lesssim\nu^2F_+(t),
\end{equation}
and
\begin{equation}\label{eq:f-absorb}
 f(t)\le C_u a(t)^{\frac{1}{2}}+C\eta_\nu^4F_+(t).
\end{equation}
Consequently, with
$E_f(T)=\int_0^T f(t)^2\dd t$ and $H_f(T)=\int_0^T F_+(t)^2\dd t$,
\begin{equation}\label{eq:forcing-open}
 E_f(T)\le C_uK_a([0,T])+C\eta_\nu^8H_f(T)
 \qquad(T>0).
\end{equation}
\end{proposition}
\begin{proof}
For each dyadic $L$, set
\[
 k_L(s)=L^{-1}\min\{L^3,s^{-\frac{3}{2}}\},\qquad
 F_t(s)=\int_0^sf(t+\tau)\dd\tau\le sF_+(t).
\]
The kernel $k_L$ is nonincreasing and has integral $3$ on
$[0,\infty)$. The absolutely convergent frequency-localized Duhamel
formula and integration by parts give
\begin{align*}
 L^{-1}\|P_Lu(t)\|_\infty
 &\lesssim\int_0^\infty k_L(s)f(t+s)\dd s\\
 &=-\int_0^\infty k_L'(s)F_t(s)\dd s
 \le F_+(t)\int_0^\infty[-s k_L'(s)]\dd s=3F_+(t).
\end{align*}
The boundary products $k_L(s)F_t(s)$ vanish at zero and infinity because $F_t(s)\le sF_*$. Hence
\begin{equation}\label{eq:future-maximal}
 \sup_LL^{-1}\|P_Lu(t)\|_\infty\lesssim F_+(t).
\end{equation}
Summing over $L\lesssim\nu$, with one extra factor of $L$ for the
gradient, proves \eqref{eq:deep-low-H}. The multiplier bounds,
\eqref{eq:refined}, \eqref{eq:future-maximal}, and Sobolev give
\begin{gather*}
 \|h\|_5^5\le\|h\|_4^2\|h\|_6^3\lesssim_B a^{\frac{1}{2}},\qquad
 \|l\|_6^3\lesssim\eta_\nu^3,\\
 \|l\|_4^2\lesssim\|\nabla l\|_2\sup_LL^{-1}\|P_Ll\|_\infty
             \lesssim\eta_\nu F_+.
\end{gather*}
Since $f\lesssim\|h\|_5^5+\|l\|_5^5$ and
$\|l\|_5^5\le\|l\|_4^2\|l\|_6^3$, these estimates give
\eqref{eq:f-absorb}. Squaring and integrating yields
\eqref{eq:forcing-open}.
\end{proof}

We retain $H_f(T)$ because the averages defining $F_+$ can extend
beyond $T$. Section~\ref{sec:interaction} combines
\eqref{eq:forcing-open} with the Morawetz inequality to close both
spacetime estimates. We first establish the coercivity needed for
that inequality.

\section{Coercivity of the interaction form}\label{sec:static}
We first establish the fixed-radius coercivity needed for the spatial lower bound in~\eqref{eq:spatial-target}. The main point is to absorb the focusing contribution with constants independent of the mass of $h$. On the concentration region, the strict ground-state thresholds give coercivity for every constant positive semidefinite matrix. A slowly varying convolved Hessian allows us to use this comparison for the interaction form. Outside that region, a weighted Sobolev inequality controls the potential term by the small exterior gradient energy.

This section also explains why a fixed radius suffices when $N(t)\sim1$. Section~\ref{sec:radius} will choose a time-dependent radius for a general quasi-soliton solution and control its additional contribution to the Morawetz identity.

\subsection{The weight and the principal form}\label{subsec:kernel-construction}
The weight must meet three requirements. Its gradient should remain bounded, so that the action is controlled by the high-frequency mass and energy. Its Hessian should be positive and vary slowly relative to its trace, so that the ground-state coercivity can absorb the focusing term on the concentration region. Finally, the density term should retain a positive lower bound that is compatible with varying the radius.

A convex interpolation between a quadratic core and the linear Morawetz weight is a natural starting point. We choose an exponentially decaying radial curvature by imposing
\[
 Rk_R''=1-k_R',\qquad k_R(0)=k_R'(0)=0.
\]
This choice gives the explicit profile
\begin{equation}\label{eq:weight-profile}
 k_R(x)=|x|-R+Re^{-|x|/R},\qquad R>0.
\end{equation}
We verify below that this interpolation has the spatial properties needed for coercivity and the density estimate. Its radius derivative will then determine the choice in Section~\ref{sec:radius}.
For $r\ge0$, the radial derivatives are
\[
 k_R'(r)=1-e^{-r/R},\qquad k_R''(r)=R^{-1}e^{-r/R}.
\]
Its radial slope increases from zero to one, so $k_R$ is convex and
$|\nabla k_R|\le1$.
For $z\ne0$, set $r=|z|$ and let $P_{z^\perp}$ denote orthogonal
projection onto the plane perpendicular to $z$. The Hessian is
\[
 D^2k_R(z)=\frac{e^{-r/R}}{R}\frac{z\otimes z}{r^2}
             +\frac{1-e^{-r/R}}{r}P_{z^\perp}.
\]
The radial eigenvalue of $D^2k_R(z)$ is $R^{-1}e^{-r/R}$,
and each of its two tangential eigenvalues is $(1-e^{-r/R})/r$.
Taking the trace gives
\[
 0\le\Delta k_R(z)=R^{-1}e^{-r/R}
                  +\frac{2(1-e^{-r/R})}{r}
       \lesssim\min\{R^{-1},r^{-1}\}.
\]
The two Hessian lower bounds we will use are, for every $\zeta\in\R^3$,
\begin{equation}\label{eq:geometry}
 \begin{aligned}
 \zeta^TD^2k_R(z)\zeta
   &\ge cR^{-1}|\zeta|^2,
   && |z|\le2R,\\
 \zeta^TD^2k_R(z)\zeta
   &\ge c|z|^{-1}|P_{z^\perp}\zeta|^2,
   && |z|\ge\frac{R}{2}.
 \end{aligned}
\end{equation}
Differentiating the displayed expression for $D^2k_R$ gives
\begin{equation}\label{eq:kernel-derivative}
 |D^3k_R|+|\nabla\Delta k_R|\lesssim R^{-1}\Delta k_R,
\end{equation}
as well as
\begin{equation}\label{eq:kernel-bounds}
 |\nabla k_R(z)|+|z||D^2k_R(z)|+|z|^2|D^3k_R(z)|\le C.
\end{equation}
All constants in these estimates are absolute and independent of $R$.
At the origin the Hessian extends
continuously with $D^2k_R(0)=R^{-1}I$; higher derivatives are understood
weakly, and the displayed bounds hold almost everywhere.

The radius derivative determines the density average used below to
choose $R(t)$. Set
\begin{equation}\label{eq:psi}
 \begin{gathered}
 \psi(s)=(1+s)e^{-s}\quad(s\ge0),\qquad\psi_R(z)=\psi(|z|/R),\\
 \partial_Rk_R(r)=\psi(r/R)-1,\qquad
 \partial_R\nabla k_R=\nabla\psi_R.
 \end{gathered}
\end{equation}
Since $\int_{\R^3}\psi(|z|)\dd z=32\pi$, define
\begin{equation}\label{eq:S}
 \Gamma_R(z)=\frac{(1+|z|/R)e^{-|z|/R}}{32\pi R^3},\qquad
 S_R(h)=\iint\Gamma_R(x-y)|h(x)|^2|h(y)|^2\dd x\dd y.
\end{equation}
The positive average $S_R$ will provide a lower bound for $Q_R$ and
control the density terms in the weighted Sobolev inequality. The
Fourier representation below also shows that $S_R$ is nonincreasing
in $R$, a property used to choose the radius in Section~\ref{sec:radius}.
We write $Q_R=Q_{k_R}$ and $\BB_R=\BB_{k_R}$.
For real $v\in H^1$, define the kinetic and nonlinear interaction terms
\begin{gather}
 \cK_R(v)=\iint v(y)^2\nabla v(x)^TD^2k_R(x-y)\nabla v(x)\dd x\dd y,
 \notag\\
 \cN_R(v)=\iint\Delta k_R(x-y)v(x)^6v(y)^2\dd x\dd y,
 \qquad \BB_R(v)=Q_R(v)+4\cK_R(v)-\frac{4}{3}\cN_R(v).
 \label{eq:real-principal}
\end{gather}

\Needspace{12\baselineskip}
The next lemma compares the density term with $S_R$ and the kinetic form. It also reduces the principal lower bound to a calculation for the modulus.

\begin{lemma}
\label{lem:kernel-properties}
For every $R>0$ and $h\in H^1(\R^3)$,
\begin{equation}\label{eq:Qcompare}
 8\pi S_R(h)\le Q_R(h)\le8\pi\|h\|_4^4.
\end{equation}
For real $v\in H^1(\R^3)$, we also have
\begin{equation}\label{eq:density-kinetic}
 8\pi S_R(v)\le Q_R(v)\le4\cK_R(v).
\end{equation}
Moreover, $\BB_R(h)\ge\BB_R(|h|)$.
\end{lemma}
\begin{proof}
\emph{Step 1: The density multipliers.} Direct differentiation gives
\[
 -\Delta^2k_R(z)=R^{-3}e^{-|z|/R}(4R/|z|-1)\in L^1(\R^3).
\]
This distribution has no point mass at the origin: $\Delta k_R$ is
bounded there and $r^2(\Delta k_R)'(r)\to0$ as $r\to0$.
Using the Fourier convention of Subsection~\ref{subsec:intro-notation},
the transforms of $e^{-r/R}/r$ and $e^{-r/R}$ give, with
$\tau_R(\xi)=4\pi^2R^2|\xi|^2$,
\begin{equation}\label{eq:SQ}
 \begin{aligned}
 S_R(h)&=\int\frac{|\widehat{|h|^2}(\xi)|^2}{(1+\tau_R(\xi))^3}\dd\xi,\\
 Q_R(h)&=8\pi\int\frac{1+2\tau_R(\xi)}{(1+\tau_R(\xi))^2}
                         |\widehat{|h|^2}(\xi)|^2\dd\xi.
 \end{aligned}
\end{equation}
The positive Fourier multipliers give \eqref{eq:Qcompare}, although
the spatial kernel defining $Q_R$ changes sign.

\emph{Step 2: Compare with the kinetic form.} Set $\rho=v^2$ and $H_R=D^2k_R$. Integration by parts gives
\[
 Q_R(v)=\iint\nabla\rho(x)^TH_R(x-y)\nabla\rho(y)\dd x\dd y
 =4\iint v(x)v(y)\nabla v(x)^TH_R(x-y)\nabla v(y)\dd x\dd y.
\]
Cauchy--Schwarz for the positive semidefinite matrix $H_R$ and
symmetry in $x,y$ give $Q_R(v)\le4\cK_R(v)$.
Together with \eqref{eq:Qcompare}, this proves
\eqref{eq:density-kinetic}.
For an approximation $v_n\to v$ in $H^1$, we have
$\nabla(v_n^2)\to\nabla(v^2)$ in $L^1$ and $v_n^2\to v^2$ in $L^2$.
Since $H_R$ and the multiplier in \eqref{eq:SQ} are bounded at each
fixed $R$, the identity and \eqref{eq:density-kinetic} extend to $H^1$.

\emph{Step 3: Pass to the modulus.} Use the bilinear decomposition of~\cite[Theorem~2.3, (2.13)]{PV}. Set $\rho=|h|^2$, $p=\Imn(\bar h\nabla h)$, $A_h(x,y)=h(y)\nabla h(x)-h(x)\nabla h(y)$ and $|b|_H^2=b^*Hb$. Since $\bar h\nabla h=\frac{1}{2}\nabla\rho+\ii p$, symmetry in $x,y$ gives
\[
 \begin{aligned}
 \iint|A_h(x,y)|_{H_R(x-y)}^2\dd x\dd y
 ={}&2\iint\rho(y)\nabla h(x)^*H_R(x-y)\nabla h(x)\dd x\dd y\\
 &-\frac{1}{2}Q_R(h)-2\iint p(x)^TH_R(x-y)p(y)\dd x\dd y.
 \end{aligned}
\]
Substitution into \eqref{eq:principal} yields
\begin{equation}\label{eq:phase-reduction}
 \begin{aligned}
 \BB_R(h)={}&2Q_R(h)+2\iint|A_h(x,y)|_{H_R(x-y)}^2\dd x\dd y
                          -\frac{4}{3}\cN_R(|h|)\\
 \ge{}&\BB_R(|h|).
 \end{aligned}
\end{equation}
For the inequality, apply the scalar diamagnetic inequality to $h(x)h(y)$ in each direction $(\omega,-\omega)\in\R^6$, with $\omega\in S^2$. The directional derivatives of this product and its modulus are $\omega\cdot A_h$ and $\omega\cdot A_{|h|}$, respectively. Diagonalizing the positive semidefinite matrix $H_R(x-y)$ gives $|A_{|h|}|_{H_R}\le|A_h|_{H_R}$ almost everywhere. The product belongs to $H^1(\R^6)$, so the Sobolev chain rule justifies this argument for $h\in H^1$.
\end{proof}

By Lemma~\ref{lem:kernel-properties}, the desired lower bound $\BB_R(h)\ge Q_R(h)$ reduces to
\[
 \cK_R(v)-\frac{1}{3}\cN_R(v)\ge0,\qquad v=|h|.
\]
We first control the exterior contribution by a weighted Sobolev inequality. Subsection~\ref{subsec:static-coercivity} then uses the strict threshold gaps on the concentration region.

\subsection{A weighted Sobolev inequality}\label{subsec:weighted-sobolev}
After normalizing the spatial scale and center, almost periodicity makes the gradient small outside a fixed ball. We will apply the following inequality to an exterior cutoff in the proof of Proposition~\ref{prop:static}. Its constant is independent of $\|v\|_2$.

\begin{proposition}\label{prop:weighted-sobolev}
For real $v\in H^1(\R^3)$ and $R>0$,
\begin{equation}\label{eq:weighted-sobolev}
 \cN_R(v)\le C\|\nabla v\|_2^4\cK_R(v),
\end{equation}
where $C$ is independent of $R$ and $\|v\|_2$.
\end{proposition}
\begin{proof}
First take real $v\in C_c^\infty(\R^3)$ and write $\rho=v^2$, $e=\|\nabla v\|_2^2$. Near pairs use the full gradient; far pairs use the tangential part supplied by the Hessian. Localization in both regimes creates zeroth-order terms. We first control their density averages by $\cK_R(v)$, with no constant depending on $\|v\|_2$.

\emph{Step 1: Coarse density averages.} Set
\[
 \M_R\rho(x)=\sup_{s\ge R}\frac1{|B(0,s)|}\int_{B(x,s)}\rho,
 \qquad\chi_R=|B(0,R)|^{-1}1_{B(0,R)},\qquad\bar\rho=\chi_R*\rho.
\]
Set $\Phi=\M_R\rho$. Ball inclusion gives $\Phi(x)\le8\Phi(y)$
whenever $|x-y|\le R$, hence $\Phi\le8\chi_R*\Phi$. Moreover,
$\int_{B(x,2s)}\bar\rho\ge\int_{B(x,s)}\rho$ for $s\ge R$, so
$\Phi\le8\M\bar\rho$, where $\M$ is the Hardy--Littlewood maximal operator.
The symmetry of $\chi_R$ and the $L^2$ maximal inequality give
\[
 \int\rho\Phi
 \le8\int\bar\rho\Phi
 \le64\int\bar\rho\,\M\bar\rho
 \lesssim\|\bar\rho\|_2^2.
\]
Moreover,
\[
 0\le\chi_R*\chi_R\le CR^{-3}1_{B(0,2R)}\le C\Gamma_R,
 \qquad
 \|\bar\rho\|_2^2
 =\iint(\chi_R*\chi_R)(x-y)\rho(x)\rho(y)\dd x\dd y.
\]
Thus \eqref{eq:S} and \eqref{eq:density-kinetic} give
\begin{equation}\label{eq:coarse-density}
 \int\rho\M_R\rho\lesssim\|\bar\rho\|_2^2
       \lesssim S_R(v)\lesssim\cK_R(v).
\end{equation}

\emph{Step 2: Near pairs.} Let $E_{\rm loc}(y)=\int_{B(y,2R)}(|\nabla v|^2+R^{-2}\rho)$. Hardy's inequality centered at $y$ gives
\[
 E_{\rm loc}(y)\le e+4\int\frac{\rho(x)}{|x-y|^2}\dd x\lesssim e.
\]
Choose a smooth cutoff $\chi_y$ equal to one on $B(y,R)$,
supported in $B(y,2R)$, with $|\nabla\chi_y|\lesssim R^{-1}$.
By Sobolev,
\[
 \int_{B(y,R)}|v|^6
 \le\|\chi_yv\|_6^6
 \lesssim E_{\rm loc}(y)^3\lesssim e^2E_{\rm loc}(y).
\]
Using $\Delta k_R\lesssim R^{-1}$, we obtain
\[
 \begin{aligned}
 &\iint_{|x-y|\le R}\Delta k_R(x-y)|v(x)|^6\rho(y)\dd x\dd y\\
 &\quad\lesssim e^2\left[
 R^{-1}\iint_{|x-y|\le2R}\rho(y)|\nabla v(x)|^2\dd x\dd y
 +R^{-3}\iint_{|x-y|\le2R}\rho(x)\rho(y)\dd x\dd y\right]\\
 &\quad\lesssim e^2\left[\cK_R(v)+\int\rho\,\M_R\rho\right]
 \lesssim e^2\cK_R(v).
 \end{aligned}
\]
Here we used \eqref{eq:geometry} for the gradient term and
\eqref{eq:coarse-density} for the density average over $B(y,2R)$.

\emph{Step 3: Far pairs.} For $|x-y|>R$, the second bound in~\eqref{eq:geometry} controls tangential derivatives. We use an anisotropic Sobolev estimate on annuli centered at $y$. Let $\ell_j=2^jR$, $j\ge0$, and set
\[
 \Omega_j(y)=\{\ell_j<|x-y|<2\ell_j\},\qquad
 \Omega_j^*(y)=\{\frac{\ell_j}{2}<|x-y|<4\ell_j\},
\]
\[
 E_j(y)=\int_{\Omega_j^*(y)}(|\nabla v|^2+\ell_j^{-2}v^2),\qquad
 \Theta_j(y)=\int_{\Omega_j^*(y)}
             (|P_{(x-y)^\perp}\nabla v|^2+\ell_j^{-2}v^2).
\]
We claim that
\begin{equation}\label{eq:annular-sobolev}
 \int_{\Omega_j(y)}|v|^6\lesssim E_j(y)\Theta_j(y)^2
                                      \le E_j(y)^2\Theta_j(y).
\end{equation}
The second inequality follows from $\Theta_j(y)\le E_j(y)$ and gives
the linear dependence on the tangential energy needed below.
To prove the first inequality, translate and apply the critical rescaling
$\widetilde v(z)=\ell_j^{\frac{1}{2}}v(y+\ell_jz)$ to the unit annulus.
On $A^*=\{\frac{1}{2}<|z|<4\}$ this rescaling gives
\[
 \begin{aligned}
 \int_{A^*}(|\nabla\widetilde v|^2+|\widetilde v|^2)\dd z
 &=E_j(y),\\
 \int_{A^*}(|P_{z^\perp}\nabla\widetilde v|^2
                  +|\widetilde v|^2)\dd z&=\Theta_j(y).
 \end{aligned}
\]

Localize to a fixed finite collection of spherical coordinate charts,
with a partition of unity supported in $A^*$. In each chart, pull
back to $(r,\theta_1,\theta_2)$ and extend the localized function
$w$ by zero to $\R^3$. The norms and dilations below are in this
Euclidean coordinate space. Set
$D_1=\|\partial_rw\|_2^2$ and
$D_{i+1}=\|\partial_{\theta_i}w\|_2^2$ for $i=1,2$.
Applying Sobolev to $w(a_1r,a_2\theta_1,a_3\theta_2)$ gives
\[
 \|w\|_6^6\lesssim
 \frac{(a_1^2D_1+a_2^2D_2+a_3^2D_3)^3}{(a_1a_2a_3)^2},
 \qquad a_i>0.
\]
For $w\ne0$, each $D_i$ is positive by compact support. Taking
$a_i=D_i^{-\frac{1}{2}}$ gives
\[
 \|w\|_6^6\lesssim D_1D_2D_3.
\]
The case $w=0$ is immediate. The fixed charts have uniformly comparable Jacobians and metrics. Derivatives of the partition and cutoffs contribute only the zeroth-order terms already included in the annular energies. Thus
\[
 D_1\lesssim E_j(y),\qquad D_2+D_3\lesssim\Theta_j(y),
 \qquad \|w\|_6^6\lesssim E_j(y)\Theta_j(y)^2.
\]
Summing over the charts and rescaling back proves
\eqref{eq:annular-sobolev}, uniformly in $j$, $y$, and $R$.

The enlarged annuli have uniformly finite overlap, and
$\ell_j^{-2}\lesssim|x-y|^{-2}$ on $\Omega_j^*(y)$.
Hardy's inequality therefore gives, uniformly in $y$,
\[
 \sum_{j\ge0}E_j(y)
 \lesssim\int\left(|\nabla v(x)|^2
                      +\frac{\rho(x)}{|x-y|^2}\right)\dd x
 \lesssim e,\qquad
 \sum_{j\ge0}E_j(y)^2\lesssim e^2.
\]
Since $\Delta k_R\lesssim|x-y|^{-1}$, the contribution from these
annuli is bounded by
\[
 C\sum_j\int\rho(y)E_j(y)^2
 \left[\ell_j^{-1}\int_{\Omega_j^*(y)}|P_{(x-y)^\perp}\nabla v|^2
             +\ell_j^{-3}\int_{\Omega_j^*(y)}\rho\right]\dd y.
\]
For the gradient term, use $E_j^2\lesssim e^2$, finite overlap, and the
second Hessian bound in \eqref{eq:geometry}:
\[
 \begin{aligned}
 &\sum_j\int\rho(y)E_j(y)^2\ell_j^{-1}
           \int_{\Omega_j^*(y)}
                  |P_{(x-y)^\perp}\nabla v(x)|^2\dd x\dd y\\
 &\quad\lesssim e^2\iint_{|x-y|>\frac{R}{2}}
       \frac{\rho(y)}{|x-y|}
                  |P_{(x-y)^\perp}\nabla v(x)|^2\dd x\dd y
 \lesssim e^2\cK_R(v).
 \end{aligned}
\]
For the density term, use
\[
 \ell_j^{-3}\int_{\Omega_j^*(y)}\rho
 \le\ell_j^{-3}\int_{B(y,4\ell_j)}\rho\lesssim\M_R\rho(y).
\]
Keeping $E_j(y)^2$ inside the sum preserves the summability supplied by Hardy's inequality:
\[
 \begin{aligned}
 &\sum_j\int\rho(y)E_j(y)^2\ell_j^{-3}
                     \int_{\Omega_j^*(y)}\rho(x)\dd x\dd y\\
 &\quad\lesssim\int\rho(y)\M_R\rho(y)\sum_jE_j(y)^2\dd y
 \lesssim e^2\int\rho\,\M_R\rho
 \lesssim e^2\cK_R(v).
 \end{aligned}
\]
Combining the near and far contributions proves
\eqref{eq:weighted-sobolev} for smooth functions. At fixed $R$,
both $D^2k_R$ and $\Delta k_R$ are bounded; approximation in $H^1$,
and hence in $L^2\cap L^6$, completes the proof.
\end{proof}

\subsection{Coercivity below the ground state}\label{subsec:static-coercivity}
The bound $\BB_R(v)\ge Q_R(v)$ is equivalent to $\cK_R(v)-\frac{1}{3}\cN_R(v)\ge0$. On the concentration region, freeze $D^2k_R*v^2$ to a constant positive semidefinite matrix. In its eigenbasis, the comparison requires control of $\|\partial_\omega v\|_2^2-\frac{1}{3}\|v\|_6^6$ in each direction. The strict threshold gaps provide that control. Outside the core, the cutoff has small $\dot H^1$ norm, so~\eqref{eq:weighted-sobolev} suffices without a sharp constant.

\Needspace{14\baselineskip}
\begin{proposition}\label{prop:static}
Let $\mathcal C\subset\dot H^1$ be compact, not containing zero, and satisfy
\[
 \sup_{z\in\mathcal C}E(z)<E(W),\qquad
 \sup_{z\in\mathcal C}\|\nabla z\|_2<\|\nabla W\|_2.
\]
There are $\epsilon_0,R_*>0$, depending only on $\mathcal C$, such
that every $h\in H^1$ with
$\inf_{z\in\mathcal C}\|h-z\|_{\dot H^1}<\epsilon_0$ satisfies
\begin{equation}\label{eq:static}
 \BB_R(h)\ge Q_R(h),\qquad R\ge R_*.
\end{equation}
The constants do not depend on the mass of $h$.
\end{proposition}
\begin{proof}
\emph{Step 1: Directional coercivity.} We derive a lower bound uniform in the direction and in the stated neighborhood. Write
\[
 \begin{gathered}
 \delta_E=E(W)-\sup_{z\in\mathcal C}E(z)>0,\qquad
 \delta_G=G-\sup_{z\in\mathcal C}\|\nabla z\|_2^2>0,\\
 b_{\mathcal C}=\inf_{z\in\mathcal C}\|z\|_6^6>0.
 \end{gathered}
\]
Here $b_{\mathcal C}>0$ follows from compactness and the exclusion of
zero. By continuity of the energy and of the $\dot H^1$ and $L^6$ norms,
we may first choose $\epsilon_0$ so that every $h$ in the stated
neighborhood satisfies
\[
 \begin{gathered}
 E(h)\le E(W)-\frac{1}{2}\delta_E,\qquad
 \|\nabla h\|_2^2\le G-\frac{1}{2}\delta_G,\\
 \|h\|_6^6\ge\frac{1}{2}b_{\mathcal C}.
 \end{gathered}
\]
We may decrease $\epsilon_0$ later without changing these bounds.
By \eqref{eq:phase-reduction}, it suffices to work with $v=|h|$.
The diamagnetic inequality gives $\|\nabla v\|_2\le\|\nabla h\|_2$
and $E(v)\le E(h)$, while $\|v\|_6=\|h\|_6$.
Let $A=\|\nabla v\|_2^2$, $B_6=\|v\|_6^6$ and, for $\omega\in S^2$,
$A_\omega=\|\partial_\omega v\|_2^2$. For $v\ne0$, $A_\omega>0$:
otherwise $v$ would be constant along the direction $\omega$, which is
incompatible with $v\in L^6$. Write $x=s\omega+y$, with $y\in\omega^\perp$, and set $v_\lambda(s,y)=\lambda^{\frac{1}{2}}v(\lambda s,y)$. This dilation leaves the transverse gradient energy unchanged and multiplies both the directional gradient energy and the sextic integral by $\lambda^2$. Hence
\[
 \|\nabla v_\lambda\|_2^2=A+(\lambda^2-1)A_\omega,\qquad
 \|v_\lambda\|_6^6=\lambda^2B_6.
\]
Choose $\lambda^2=1+(G-A)/A_\omega$, so that
$\|\nabla v_\lambda\|_2^2=G$. Sharp Sobolev gives
$\lambda^2B_6\le G$, while $B_6\le A^3/G^2<G$. Rearranging,
\begin{equation}\label{eq:directional}
 A_\omega-\frac{1}{3}B_6
 \ge\frac{(G-A)B_6}{G-B_6}-\frac{1}{3}B_6
 =\frac{2B_6(E(W)-E(v))}{G-B_6}\ge\sigma>0.
\end{equation}
Here $\sigma=b_{\mathcal C}\delta_E/(2G)>0$ is independent of $h$ and
$\omega$.
For a constant positive semidefinite matrix $P$, take an orthonormal eigenbasis $\omega_1,\omega_2,\omega_3$ with eigenvalues $\lambda_j\ge0$. Multiplying~\eqref{eq:directional} for $\omega_j$ by $\lambda_j$ and summing gives
\begin{equation}\label{eq:matrix-core}
 \int\nabla v^TP\nabla v-\frac{1}{3}\tr(P)\int v^6\ge\sigma\tr(P).
\end{equation}

\emph{Step 2: Freeze the convolved Hessian.} We work on a large ball centered at the origin. Set $\mathsf A_R=D^2k_R*v^2$, $w_R=\tr\mathsf A_R$ and $W_0=w_R(0)>0$. We compare $\mathsf A_R(x)$ on this ball with the constant matrix $\mathsf A_R(0)$, to which~\eqref{eq:matrix-core} applies. Convolving~\eqref{eq:kernel-derivative} with the full density $v^2$ gives
\[
 |\nabla w_R|+\sum_{j=1}^3\|\partial_j\mathsf A_R\|_{\rm op}\lesssim R^{-1}w_R.
\]
Integrating along line segments, first for $w_R$ and then for
$\mathsf A_R$, yields
\begin{equation}\label{eq:freezing}
 w_R(x)\le e^{C|x|/R}W_0,\qquad
 \|\mathsf A_R(x)-\mathsf A_R(0)\|_{\rm op}
                 \lesssim(e^{C|x|/R}-1)W_0.
\end{equation}
The frozen form has margin $\sigma W_0$, while~\eqref{eq:freezing} gives an error $o_R(1)W_0$ on each fixed ball. No uniform lower bound for the smallest eigenvalue is needed: the coercivity margin and the freezing error are both measured against the trace. This comparison keeps the constants independent of the mass of $h$. Let $\delta>0$ be chosen below. Compactness of $\mathcal C$ and Hardy's inequality give
\[
 \sup_{z\in\mathcal C}\int_{|x|>C_0}
       (|\nabla z|^2+|x|^{-2}|z|^2)\longrightarrow0
 \qquad(C_0\to\infty).
\]
Choose $C_0\ge1$ and then decrease $\epsilon_0$.
For $z\in\mathcal C$ with $\|h-z\|_{\dot H^1}<\epsilon_0$, Hardy's
inequality controls $|x|^{-1}(h-z)$ by this norm. The preceding
tail bound and the diamagnetic inequality then give
\begin{equation}\label{eq:exterior-small}
 \int_{|x|>C_0}(|\nabla v|^2+|x|^{-2}v^2)\le\delta.
\end{equation}
Take $0\le\chi\le1$, zero on $B_{C_0}$ and one outside $B_{2C_0}$,
with $|\nabla\chi|\lesssim C_0^{-1}$. Put
\[
 v_e=\chi v,\qquad \rho_c=v^2-v_e^2\ge0,\qquad
 K_e=\int_{|x|>C_0}\nabla v^T\mathsf A_R\nabla v.
\]
Then $\|\nabla v_e\|_2^2\lesssim\delta$. The term $K_e$ retains the original gradient on the exterior region and the full density $v^2$ in $\mathsf A_R$; the form $\cK_R(v_e)$ also differentiates the cutoff and uses the density $v_e^2$.
Since $(1-\chi^6)v^6$ is supported in $B_{2C_0}$,
\eqref{eq:freezing} bounds the freezing error by $\epsilon_RW_0$, where
\[
 \epsilon_R=C(e^{2CC_0/R}-1)(A+B_6)\longrightarrow0
 \qquad(R\to\infty),
\]
uniformly over the stated neighborhood.
Using \eqref{eq:matrix-core}, \eqref{eq:exterior-small}, and
$\|\mathsf A_R(0)\|_{\rm op}\le W_0$, we obtain
\[
 \begin{aligned}
 &\int_{|x|\le C_0}\nabla v^T\mathsf A_R\nabla v
       -\frac{1}{3}\int(1-\chi^6)w_Rv^6\\
 &\quad\ge\int\nabla v^T\mathsf A_R(0)\nabla v
       -\frac{W_0}{3}\int v^6
       -W_0\int_{|x|>C_0}|\nabla v|^2-\epsilon_RW_0\\
 &\quad\ge(\sigma-\delta-\epsilon_R)W_0.
 \end{aligned}
\]
In the first inequality we discarded the nonnegative term
$\frac{W_0}{3}\int\chi^6v^6$.
For $\delta\le\frac{\sigma}{4}$ and $\epsilon_R\le\frac{\sigma}{4}$,
\eqref{eq:real-principal} gives
\begin{equation}\label{eq:before-tail}
 \BB_R(v)\ge Q_R(v)+2\sigma W_0+4K_e-\frac{4}{3}\int w_Rv_e^6.
\end{equation}

\emph{Step 3: Absorb the exterior potential.} The last term of~\eqref{eq:before-tail} remains to be controlled. Split the density as $v^2=v_e^2+\rho_c$:
\[
 \int w_Rv_e^6=\cN_R(v_e)+\int(\Delta k_R*\rho_c)v_e^6.
\]
Since $0\le\rho_c\le v^2$ is supported in $B_{2C_0}$, the kernel
bounds give, for $R\ge2C_0$,
\[
 \|\Delta k_R*\rho_c\|_\infty
 \lesssim R^{-1}\int\rho_c
 \lesssim\int_{B_{2C_0}}\Delta k_R(-y)v(y)^2\dd y\le W_0.
\]
Proposition~\ref{prop:weighted-sobolev} and the Sobolev inequality yield
\[
 \begin{aligned}
 \int w_Rv_e^6
 &\le C\|\nabla v_e\|_2^4\cK_R(v_e)+CW_0\|v_e\|_6^6\\
 &\le C\delta^2\cK_R(v_e)+C\delta^3W_0.
 \end{aligned}
\]
Using $v_e^2\le v^2$, the positivity of $D^2k_R$, and
\eqref{eq:freezing}, we obtain
\[
 \begin{aligned}
 \cK_R(v_e)
 &\le\int\nabla(\chi v)^T\mathsf A_R\nabla(\chi v)\\
 &\le2K_e+2\int v^2\nabla\chi^T\mathsf A_R\nabla\chi\\
 &\le2K_e+Ce^{2CC_0/R}W_0
          \int_{C_0<|x|<2C_0}|x|^{-2}v^2
 \le2K_e+C\delta W_0.
 \end{aligned}
\]
For $R\ge2C_0$ the exponential factor is bounded by an absolute
constant. Hence
\[
 \int w_Rv_e^6\le C_{\rm tail}\delta^2K_e+C_{\rm tail}\delta^3W_0,
\]
with $C_{\rm tail}$ absolute. The exterior kinetic term absorbs the multiple of $K_e$, and the directional margin absorbs the multiple of $W_0$. Choose $\delta>0$ so that
\[
 \delta\le\frac{\sigma}{4},\qquad
 \frac{4}{3}C_{\rm tail}\delta^2\le2,\qquad
 \frac{4}{3}C_{\rm tail}\delta^3\le\sigma.
\]
The order of choices is $\sigma$, then $\delta$, then $C_0$ and the
final $\epsilon_0$, and finally $R_*$, with $R_*\ge2C_0$ and
$\epsilon_R\le\frac{\sigma}{4}$ for every $R\ge R_*$. Substituting into
\eqref{eq:before-tail} gives
\[
 \BB_R(v)\ge Q_R(v)+2K_e+\sigma W_0\ge Q_R(v),
 \qquad R\ge R_*.
\]
Since $v=|h|$, we have $Q_R(v)=Q_R(h)$.
The lower bound for $v$ and \eqref{eq:phase-reduction} therefore prove
\eqref{eq:static} for $h$.
\end{proof}

\subsection{The soliton-like case.}
To illustrate the static estimate, suppose $1\le N(t)\le C_u$. Fix $R\ge R_*$ and choose
$\nu$ sufficiently small. Proposition~\ref{prop:static}, applied
after translation and critical rescaling, gives
$\BB_R(h(t))\ge Q_R(h(t))$ since $N(t)R\ge R_*$.
The concentration bound \eqref{eq:a-bounds} and the uniform
$L^6$ bound give constants $R_0,m_0>0$, independent of $t$ and
small $\nu$, such that
\[
 \int_{B(x(t),R_0)}|h(t,x)|^2\dd x\ge m_0,
 \qquad
 Q_R(h(t))\ge8\pi\Gamma_R(2R_0)m_0^2\eqqcolon q_R>0.
\]
Here we used \eqref{eq:Qcompare} and the positivity and radial
monotonicity of $\Gamma_R$; the moving center causes no change.
Lemma~\ref{lem:l4} and \eqref{eq:Fstar} give
$a+f^2+F_+^2\lesssim_u1$. Since the radius is constant, the source
estimate of Proposition~\ref{prop:full-source}, proved below, and
the integrated interaction identity give
\[
 q_RT\le C_u\nu^{-3}+C_ud_\nu^{\frac{1}{2}}T.
\]
Choose $\nu$ so that $C_ud_\nu^{\frac{1}{2}}\le \frac{q_R}{2}$ and let
$T\to\infty$, obtaining a contradiction.

\section{A density-dependent radius}\label{sec:radius}
In the general quasi-soliton scenario, $N(t)$ may be unbounded. A fixed radius still gives coercivity, but its density term need not control $a(t)$ uniformly. We therefore choose $R(t)$ from the density so that both $N(t)R(t)\ge R_*$ and $Q_{R(t)}(h(t))\gtrsim_u a(t)$ hold. The continuity equation will control the additional term $R'\partial_RM_{k_R}$ in~\eqref{eq:full-identity}.

\begin{proposition}[Estimates for the interaction weight]\label{prop:weight-estimates}
Let $u$ be the solution fixed in Section~\ref{sec:reduction}.
There exist $c_u,C_u,\nu_0>0$ such that, for every $0<\nu\le\nu_0$,
one can choose a locally Lipschitz radius
$R:[0,\infty)\to(0,\infty)$ with
$R(t)\sim_u a(t)$. Set
\[
 k(t,x)=k_{R(t)}(x),
 \qquad M(t)=M_{k(t)}(h(t)).
\]
With $D(t)=R'(t)\partial_RM_{k_R}(h(t))|_{R=R(t)}$, we have
\[
 \BB_{k(t)}(h(t))\ge c_u a(t),\qquad
 D(t)\ge-C_u(\eta_\nu^2+\nu)a(t)
\]
for almost every $t\ge0$. After decreasing $\nu_0$, these imply
\[
 \BB_{k(t)}(h(t))+D(t)\ge\frac{1}{2}c_u a(t).
\]
The same weight satisfies
\[
 |\nabla k(t,x)|+|x||D^2k(t,x)|+|x|^2|D^3k(t,x)|\le C,
 \qquad |M(t)|\le C_u\nu^{-3},
\]
with $C$ absolute. The spatial derivatives are understood weakly,
and their bounds hold almost everywhere in $x$.
The constants in these inequalities and in $R(t)\sim_u a(t)$ are
independent of $\nu$, $t$, and the terminal time.
\end{proposition}

\subsection{Choice of the radius}\label{subsec:radius-choice}
The comparison $Q_R\ge8\pi S_R$ suggests selecting $R$ through the normalized density average $S_R$. We use the kernel $\psi_R=\partial_Rk_R+1$ because differentiating its density pairing in time produces $\partial_RM_{k_R}$ through the continuity equation. This compatibility will control the moving-radius term by completing a square, without the scale-smoothing argument used in~\cite[Section~6]{Dodson4}.

For fixed $0\ne h\in H^1$, write $\rho=|h|^2$ and choose $R$ by
\[
 \iint\psi_R(x-y)\rho(x)\rho(y)\dd x\dd y
 =\frac{32\pi}{\Lambda}R^4,
\]
where $\Lambda\ge1$ will be fixed below. Under the critical scaling
$h_\lambda(x)=\lambda^{\frac{1}{2}}h(\lambda x)$ and
$R\mapsto R/\lambda$, both sides are multiplied by $\lambda^{-4}$.
Since the left-hand side equals $32\pi R^3S_R(h)$, the equation is
equivalent to
\begin{equation}\label{eq:radius-equation}
 R=\Lambda S_R(h).
\end{equation}
By \eqref{eq:Qcompare}, this gives $Q_R(h)\ge8\pi R/\Lambda$.
The next lemma compares $R$ with $a=\|h\|_4^4$ and hence gives the
required quartic lower bound.

With $h$ fixed, the multiplier in~\eqref{eq:SQ} shows that the normalized average satisfies $\partial_RS_R(h)\le0$. Thus $\partial_R[R-\Lambda S_R(h)]\ge1$. This monotonicity concerns the normalized average, not the unnormalized density pairing.
Moreover, $S_R\to\|h\|_4^4$ as $R\to0$, and $S_R\le\|h\|_4^4$ for all
$R>0$. For $a=\|h\|_4^4>0$, the function $R-\Lambda S_R(h)$ is
negative near zero and nonnegative at $R=\Lambda a$.
Its strict monotonicity therefore gives a unique positive root of
\eqref{eq:radius-equation}.

\begin{lemma}\label{lem:radius}
Let $0\ne h\in H^1(\R^3)$, set $a=\|h\|_4^4$, and assume
$\|\nabla h\|_2\le B$ and $\Lambda\ge1$.
The root of \eqref{eq:radius-equation} satisfies
\begin{equation}\label{eq:radius-bounds}
 c(1+B)^{-3}\Lambda^{\frac{1}{4}}a\le R\le\Lambda a.
\end{equation}
Here $c>0$ is an absolute constant.
\end{lemma}
\begin{proof}
Put $\rho=|h|^2$. The translation estimate
$\|h(\cdot-z)-h\|_2\le |z|\|\nabla h\|_2$, interpolation between
$L^2$ and $L^6$, and Sobolev give
\[
 \|\rho(\cdot-z)-\rho\|_2
 \le2\|h\|_6\|h(\cdot-z)-h\|_3
 \lesssim B^2|z|^{\frac{1}{2}}.
\]
Since $\Gamma_s$ has integral one and first absolute moment $\frac{15s}{4}$,
\[
 a-S_s(h)=\frac{1}{2}\int\Gamma_s(z)
       \|\rho(\cdot-z)-\rho\|_2^2\dd z\le CB^4s.
\]
Take $C\ge1$ and $s_0=\frac{a}{2C(1+B)^4}$, so $S_s\ge\frac{a}{2}$ for $s\le s_0$.
If the root satisfied $R<s_0$, then $R=\Lambda S_R\ge\frac{\Lambda a}{2}\ge s_0$, a contradiction. Hence $R\ge s_0$. Since $\psi$ decreases in its argument, the unnormalized kernel $\psi_R$ dominates $\psi_{s_0}$ at every fixed separation. Dividing by the normalization factors gives
\[
 \Gamma_R\ge(s_0/R)^3\Gamma_{s_0},\qquad
 R=\Lambda S_R\ge\frac{\Lambda a s_0^3}{2R^3}.
\]
Rearrangement proves the lower bound; the upper bound follows from
$S_R\le a$.
\end{proof}

We now record the principal lower bound used in the integrated
Morawetz estimate.

\Needspace{18\baselineskip}
\begin{proposition}[Coercivity for the chosen radius]\label{prop:chosen-positive}
Let $u$ be the critical element fixed in Section~\ref{sec:reduction},
and let $h$ and $a$ be as in \eqref{eq:fixed-notation}.
There exist $\Lambda\ge1$, $\nu_0>0$, and $\kappa_u,c_u>0$, depending
only on the compact closure of the normalized orbit and its threshold gaps,
such that
for $0<\nu\le\nu_0$ the radius $R(t)$ defined by
\eqref{eq:radius-equation} satisfies, for every $t\ge0$,
\[
 \kappa_u a(t)\le R(t)\le\Lambda a(t),\qquad N(t)R(t)\ge R_*,
\]
and
\begin{equation}\label{eq:chosen-positive}
 \begin{aligned}
 \BB_{R(t)}(h(t))\ge Q_{R(t)}(h(t))
       &\ge8\pi S_{R(t)}(h(t))\\
       &=\frac{8\pi R(t)}{\Lambda}\ge c_u a(t).
 \end{aligned}
\end{equation}
Here $R_*$ is the constant in Proposition~\ref{prop:static} for
$\mathcal C=\overline{\Kcal}$. All constants are independent of $\nu$ and $t$.
\end{proposition}
\begin{proof}
The compact set $\overline{\Kcal}$ excludes zero and satisfies the
strict thresholds by \eqref{eq:trapping} and energy conservation.
Apply Proposition~\ref{prop:static} to obtain $\epsilon_0$ and $R_*$.
For $\Lambda\ge1$, set $\kappa(\Lambda)=c(1+B)^{-3}\Lambda^{\frac{1}{4}}$.
Choose $\Lambda$ so that $\kappa(\Lambda)c_{\Kcal}\ge R_*$, and set
$\kappa_u=\kappa(\Lambda)$.
By Proposition~\ref{prop:frequency-properties}, we may then choose
$\nu_0$ so that, for $0<\nu\le\nu_0$,
\[
 \sup_t\inf_{z\in\overline{\Kcal}}
       \|\widetilde h(t)-z\|_{\dot H^1}\le\eta_\nu<\epsilon_0,
 \qquad N(t)a(t)\ge c_{\Kcal}>0.
\]
Thus $h(t)\ne0$ and the radius is defined. Equation~\eqref{eq:radius-bounds} gives
\[
 \kappa_u a(t)\le R(t)\le\Lambda a(t),\qquad
 N(t)R(t)\ge\kappa_u N(t)a(t)\ge\kappa_u c_{\Kcal}\ge R_*.
\]

Translation invariance and critical scaling give
\[
 \mathcal A_{R(t)}(h(t))
 =N(t)^{-1}\mathcal A_{N(t)R(t)}(\widetilde h(t)),
 \qquad\mathcal A\in\{S,Q,\BB\}.
\]
Apply \eqref{eq:static} to $\widetilde h(t)$ at radius $N(t)R(t)$ and scale back.
Equation~\eqref{eq:Qcompare} and the radius equation then give
\eqref{eq:chosen-positive}, with $c_u=8\pi\kappa_u/\Lambda$.
\end{proof}

The cutoff $\nu$ may be decreased further to meet the absorption
conditions in Section~\ref{sec:interaction}, while $\Lambda$ and $c_u$
remain fixed. The construction gives $R(t)\sim_u a(t)$ and a lower bound for $N(t)R(t)$; it does not require $R(t)\sim_u N(t)^{-1}$.

\subsection{The contribution of the moving radius}\label{subsec:radius-variation}
Proposition~\ref{prop:chosen-positive} controls the principal term. We now establish local Lipschitz regularity of $R(t)$ and use the continuity equation to bound the negative part of $D(t)$ in \eqref{eq:full-identity}. Recall $g=P_{\le\nu}F(u)$ from Section~\ref{sec:lp} and $p=\Imn(\bar h\nabla h)$ from Subsection~\ref{subsec:interaction-identity}.
The equation for $h$ gives
\begin{equation}\label{eq:continuity}
 \partial_t|h|^2=-2\Div p+\mathfrak r,
 \qquad \mathfrak r=-2\Imn(\mathcal E\bar h)
                  =2\Imn(F(u)\bar l+g\bar h).
\end{equation}
H\"older's inequality, Sobolev, and Bernstein, together with
$f\le F_*$, yield
\begin{equation}\label{eq:mass-source}
 \|\mathfrak r(t)\|_1
 \lesssim\|F(u)\|_{\frac{6}{5}}\|l\|_6+\|g\|_{\frac{6}{5}}\|h\|_6
 \lesssim_u\eta_\nu+\nu^{\frac{1}{2}}f(t)
 \lesssim_u\eta_\nu+\nu^{\frac{1}{2}}.
\end{equation}
The continuity equation holds distributionally; its fixed-radius
integrated form is justified in the proof below.
Using \eqref{eq:psi} and $\rho(t)=|h(t)|^2$, define, for $t\ge0$ and $R>0$,
\[
 \begin{aligned}
 M_\psi(t,R)&=\partial_RM_{k_R}(h(t))
       =2\int p(t)\cdot(\nabla\psi_R*\rho(t)),\\
 S_{\rm src}(t,R)&=\int\mathfrak r(t)(\psi_R*\rho(t)).
 \end{aligned}
\]
The radius derivative holds $h(t)$ fixed. Since $\psi_R,\rho(t)\in L^2$, their convolution is bounded, allowing the $L^1$ mass-source estimate \eqref{eq:mass-source} to control the source term.

\Needspace{10\baselineskip}
\begin{lemma}\label{lem:radius-variation}
For $0<\nu\le\nu_0$ as in Proposition~\ref{prop:chosen-positive},
the radius $R:[0,\infty)\to(0,\infty)$ determined by
\eqref{eq:radius-equation} is Lipschitz on every compact time interval
and satisfies, for almost every $t\ge0$,
\begin{equation}\label{eq:radius-derivative}
 \begin{aligned}
 c_RR'&=2\big[M_\psi(t,R(t))+S_{\rm src}(t,R(t))\big],\\
 c_R&=\frac{32\pi R^3}{\Lambda}(1-\Lambda\partial_RS_R(h(t)))
                                     \ge\frac{32\pi R^3}{\Lambda}.
 \end{aligned}
\end{equation}
Here $c_R$ is evaluated at $R=R(t)$. Consequently,
\begin{equation}\label{eq:radius-error}
 \begin{aligned}
 D(t)=R'M_\psi(t,R(t))
 &\ge-\frac{7\Lambda}{64}a(t)\|\mathfrak r(t)\|_1^2\\
 &\ge-C_u(\eta_\nu^2+\nu)a(t).
 \end{aligned}
\end{equation}
\end{lemma}
\begin{proof}
\emph{Step 1: Local regularity of the radius.} Fix $\nu$ and a compact interval $I\subset[0,\infty)$. Since $h\in C(I;H^1)$ and $a(t)>0$ on $I$, Proposition~\ref{prop:chosen-positive} gives
\[
 0<R_-\coloneqq\kappa_u\min_{t\in I}a(t)
 \le R(t)\le
 R_+\coloneqq\Lambda\max_{t\in I}a(t)<\infty.
\]
For this fixed $\nu$, $\mathcal E\in C(I;L^{\frac{6}{5}})$, so
$\rho,p,\mathfrak r\in C(I;L^1)$ by H\"older's inequality and
continuity of the defining products.
At each fixed $R\in[R_-,R_+]$, symmetry and
\eqref{eq:continuity} give
\[
 \partial_t\iint\psi_R(x-y)\rho(x)\rho(y)\dd x\dd y
                   =2\big[M_\psi(t,R)+S_{\rm src}(t,R)\big].
\]
To justify this identity, spatially mollify the projected equation
for $h$ as in Subsection~\ref{subsec:interaction-identity}.
The resulting densities and mass sources converge in
$C(I;L^1)$, while $\psi_R$ and $\nabla\psi_R$ are uniformly bounded
for $R\in[R_-,R_+]$. Passing to the integrated identity gives the
displayed formula. Its right-hand side is continuous in $(t,R)$,
since the densities and source are continuous in $L^1$ and the
kernels depend continuously on $R$ in $L^\infty$.
Hence, uniformly for $t\in I$ and $R\in[R_-,R_+]$,
\[
 |\partial_tS_R(h(t))|
 =\frac{|M_\psi(t,R)+S_{\rm src}(t,R)|}{16\pi R^3}
 \le C_{I,\nu}.
\]
Comparing the radius equations at two times and using
$\partial_R[R-\Lambda S_R(h)]\ge1$, we obtain
\[
 |R(t)-R(s)|\le\Lambda\sup_{r\in[R_-,R_+]}
 |S_r(h(t))-S_r(h(s))|\le C_{I,\nu}|t-s|.
\]

\emph{Step 2: The radius-motion term.} Differentiating the explicit kernel also shows that $\partial_RS_R(h(t))$ is jointly continuous in $(t,R)$ on $I\times[R_-,R_+]$. Together with Step~1, this justifies the chain rule along the locally Lipschitz radius. Apply it to $32\pi R(t)^3S_{R(t)}(h(t))=\frac{32\pi R(t)^4}{\Lambda}$. Since $S_R(h(t))=R/\Lambda$ at $R=R(t)$, the coefficient of $R'$ is
\[
 \frac{128\pi R^3}{\Lambda}
       -\partial_R(32\pi R^3S_R)
 =\frac{32\pi R^3}{\Lambda}(1-\Lambda\partial_RS_R)=c_R,
\]
which proves \eqref{eq:radius-derivative} almost everywhere.
Evaluate the remaining radius-dependent quantities at $R=R(t)$ and suppress the arguments of $M_\psi$ and $S_{\rm src}$. Completing the square gives
\[
 D=R'M_\psi
 =\frac{c_R}{2}\left(R'-\frac{S_{\rm src}}{c_R}\right)^2
      -\frac{S_{\rm src}^2}{2c_R}
 \ge-\frac{S_{\rm src}^2}{2c_R}.
\]
The negative part depends only on the mass source, so no separate bound for $|R'|$ is needed. Young's inequality and Cauchy--Schwarz give
\[
 \begin{aligned}
 |S_{\rm src}|^2
 &\le\|\mathfrak r\|_1^2\|\psi_R*\rho\|_\infty^2\\
 &\le\|\mathfrak r\|_1^2\|\psi_R\|_2^2\|\rho\|_2^2\\
 &=7\pi R^3a\|\mathfrak r\|_1^2.
 \end{aligned}
\]
Together with $c_R\ge32\pi R^3/\Lambda$ and \eqref{eq:mass-source},
this proves \eqref{eq:radius-error}.
\end{proof}

\begin{proof}[Proof of Proposition~\ref{prop:weight-estimates}]
Choose $\Lambda$ and $\nu_0$ as in
Proposition~\ref{prop:chosen-positive}, and determine $R(t)$ by
\eqref{eq:radius-equation}. Lemmas~\ref{lem:radius}
and~\ref{lem:radius-variation} give $R(t)\sim_u a(t)$ and local
Lipschitz regularity. The two lower bounds follow from
\eqref{eq:chosen-positive} and \eqref{eq:radius-error};
since $\eta_\nu\to0$, decreasing $\nu_0$ gives the combined lower bound.
The kernel estimate is \eqref{eq:kernel-bounds}, and the action bound
follows from $|\nabla k_R|\le1$ and the estimate in
Subsection~\ref{subsec:interaction-identity}.
\end{proof}

\section{Frequency-truncation errors}\label{sec:source}
We estimate the source term in \eqref{eq:source-identity} on $[0,T]$,
retaining $H_f(T)$ so that the bound can be combined with
\eqref{eq:forcing-open}. Write $\mathcal E=G_1-g$, where
$G_1=F(u)-F(h)$ and $g=P_{\le\nu}F(u)$.

For $G_1$, cancellations expose the low-frequency factors estimated
in Section~\ref{sec:lp}. For $g$, the identity
$h=\Div(\nabla\Delta^{-1}h)$ transfers derivatives to the projected
source and the convolution fields. Long-time Strichartz and
Bernstein estimates then give a positive power of $d_\nu$ in each
term. The proof below uses the cancellations and integrations by parts
from \cite[Lemmas~6.7--6.8]{KVquintic}, together with the low-frequency
bounds of Section~\ref{sec:lp}. In this form, the estimates require
only the spatial bounds \eqref{eq:kernel-source-bounds} and therefore
apply uniformly to $k_{R(t)}$.

\begin{proposition}\label{prop:full-source}
Fix $C_k>0$. Let $k(t,x)$ be real and even in $x$, measurable in $t$,
and satisfy \eqref{eq:kernel-source-bounds} uniformly in $t$, with
constant $C_k$ and with the spatial derivatives understood in the weak sense.
There is $\nu_0=\nu_0(u)>0$ such that, for every
$0<\nu\le\nu_0$ and $T>0$, the following estimate holds:
\begin{equation}\label{eq:source-final}
 \left|\int_0^T\mathrm{Err}_{k,\nu}(t)\dd t\right|
 \le C_{u,C_k}d_\nu^{\frac{1}{2}}
               [\nu^{-3}+K_a([0,T])+H_f(T)].
\end{equation}
The constant $C_{u,C_k}$ is independent of $T$, $\nu$, and the
individual kernel in this class.
\end{proposition}
\begin{proof}
All integrations by parts are spatial, so the kernel need only be measurable in $t$. Spacetime norms are over $[0,T]\times\R^3$ unless specified otherwise.

\emph{Step 1: The convolution fields.} Begin with $q$ and $q_m$ from Subsection~\ref{subsec:interaction-identity}. The following spatial estimates hold for almost every time.
Interpolation between $L^2$ and $L^6$ gives
\[
 \|h\|_3\lesssim M_\nu^{\frac{1}{2}}B^{\frac{1}{2}},\qquad
 \||h|^2\|_{\frac{6}{5}}\lesssim M_\nu^{\frac{3}{2}}B^{\frac{1}{2}}.
\]
Boundedness of $\nabla k$ controls $q$ and $q_m$ through the mass and momentum densities. Hardy's inequality controls $\Div q$ in the supremum norm. For the differentiated fields, use~\eqref{eq:kernel-source-bounds} and the Hardy--Littlewood--Sobolev estimates
\[
 \big\||x|^{-1}*b\big\|_6+\big\||x|^{-2}*b\big\|_2
 \lesssim\|b\|_{\frac{6}{5}},\qquad b\in L^{\frac{6}{5}}(\R^3),
\]
applied to $\rho$ for $Dq$, $\Div q$ and $\nabla\Div q$, and to the components of $p$ for $\nabla q_m$. Estimate the momentum input by H\"older, using the preceding $L^3$ bound for $h$. Together with the interpolation bounds, this gives
\begin{align}
 \|q\|_\infty&\lesssim_{C_k} M_\nu^2,&
 \|q_m\|_\infty+\|\Div q\|_\infty&\lesssim_{C_k} M_\nu B,\notag\\
 \|\nabla q_m\|_6&\lesssim_{C_k} M_\nu^{\frac{1}{2}}B^{\frac{3}{2}},&
 \|Dq\|_6+\|\Div q\|_6+\|\nabla\Div q\|_2
                         &\lesssim_{C_k} M_\nu^{\frac{3}{2}}B^{\frac{1}{2}}.
 \label{eq:qbounds}
\end{align}

\emph{Step 2: The nonlinear difference.} Let $\mathrm{Err}_{G_1}$ and $\mathrm{Err}_g$ be the terms from~\eqref{eq:source-identity} with $\mathcal E$ replaced by $G_1$ and $g$. Thus $\mathrm{Err}_{k,\nu}=\mathrm{Err}_{G_1}-\mathrm{Err}_g$. The first term has two cancellations:
\begin{equation}\label{eq:mass-cancel}
 \Imn(G_1\bar h)=-\Imn(F(u)\bar l),
\end{equation}
\begin{align}
 \int q\cdot\{G_1,h\}_p
 ={}&\int\Div q\left[\frac{2}{3}(|u|^6-|h|^6)
                         -\Ren(l\overline{F(u)})\right]\notag\\
 &-2\int q\cdot\Ren(F(u)\nabla\bar l).
 \label{eq:Gcancel}
\end{align}
The identity $\{F(v),v\}_p=-\frac{2}{3}\nabla|v|^6$ gives
\[
 \{G_1,h\}_p
 =-\nabla\left[\frac{2}{3}(|u|^6-|h|^6)-\Ren(l\overline{F(u)})\right]
       -2\Ren(F(u)\nabla\bar l).
\]
Pairing with $q$ and integrating by parts proves \eqref{eq:Gcancel}.
The $L^5$ multiplier bound gives $\|F(h)\|_1\lesssim f(t)$.
Using this and \eqref{eq:deep-low-H}, we obtain
\[
 \left\|\frac{2}{3}(|u|^6-|h|^6)-\Ren(l\overline{F(u)})\right\|_1
       \lesssim\|l\|_\infty f(t)\lesssim\nu F_+(t)f(t).
\]
By \eqref{eq:source-identity}, \eqref{eq:mass-cancel}--\eqref{eq:Gcancel}, and
the bounds \eqref{eq:deep-low-H} and \eqref{eq:qbounds},
\[
 \begin{aligned}
 |\mathrm{Err}_{G_1}(t)|
 &\lesssim(\|\Div q\|_\infty+\|q_m\|_\infty)\|l\|_\infty f(t)
           +\|q\|_\infty\|\nabla l\|_\infty f(t)\\
 &\lesssim_{u,C_k}(d_\nu+d_\nu^2)f(t)F_+(t).
 \end{aligned}
\]
Since $f\le F_+$ almost everywhere,
\begin{equation}\label{eq:G1-source}
 \left|\int_0^T\mathrm{Err}_{G_1}(t)\dd t\right|
                     \lesssim_{u,C_k}(d_\nu+d_\nu^2)H_f(T).
\end{equation}

\emph{Step 3: Long-time Strichartz and projected-source bounds.} Use long-time Strichartz from~\cite[Theorem~4.1 and (4-5)]{KVquintic}. As explained at the beginning of Section~4 there, the proof also applies to almost-periodic focusing solutions with uniformly bounded $\dot H^1$ norm, as supplied by~\eqref{eq:trapping}. For $L>0$ and a compact interval $I\subset[0,\infty)$, set $K_N(I)\coloneqq\int_I N(t)^{-1}\dd t$. The estimate is
\begin{equation}\label{eq:lts}
 \|\nabla P_{\le L}u\|_{L^2_tL^6_x(I)}
       \lesssim_u1+L^{\frac{3}{2}}K_N(I)^{\frac{1}{2}}.
\end{equation}
The cited theorem gives \eqref{eq:lts} on finite unions of characteristic intervals. For an arbitrary compact $I$, apply it to the complete characteristic intervals contained in $I$; local constancy gives $\sup_I N<\infty$, so there are finitely many. At most two endpoint pieces remain. Uniform $L^6$ boundedness of $P_{\le L}$ and the local bound in Proposition~\ref{prop:reduction} control each by a constant depending only on $u$. Since the scale integral over the complete intervals is at most $K_N(I)$, this proves \eqref{eq:lts}, in particular on $[0,T]$, without assuming $K_N([0,\infty))<\infty$.

Set $Z=\nabla\Delta^{-1}h$, so that $\Div Z=h$, and write $X=\nu^{-3}+K_a([0,T])$. On a frequency annulus, the multiplier defining $Z$ contributes $L^{-2}$ relative to $\nabla P_Lu$. The low-pass estimate also controls each annular projection at a comparable cutoff. Thus \eqref{eq:lts}, the concentration bound \eqref{eq:a-bounds}, and dyadic summation give
\begin{align}
 \|Z\|_{L^2_tL^6_x([0,T])}
 &\lesssim_u\sum_{L\gtrsim\nu}L^{-2}
                (1+L^{\frac{3}{2}}K_N([0,T])^{\frac{1}{2}})\notag\\
 &\lesssim_u\nu^{-2}+\nu^{-\frac{1}{2}}K_a([0,T])^{\frac{1}{2}}
 \lesssim_u\nu^{-\frac{1}{2}}X^{\frac{1}{2}}.
 \label{eq:Z}
\end{align}
Bernstein's inequality gives
\begin{equation}\label{eq:bernstein-source}
 \|D^mg\|_{L^2_tL^p_x([0,T])}
 \lesssim\nu^{m+3(1-\frac{1}{p})}E_f(T)^{\frac{1}{2}},
 \qquad m=0,1,2,\quad1\le p\le\infty.
\end{equation}
\emph{Step 4: Estimate the projected-source contribution.} Integration by parts with $h=\Div Z$ gives
\begin{align*}
 \int q_m\{g,h\}_m
 &=-\Imn\int\overline{Z_k}(g\partial_kq_m+q_m\partial_kg),\\
 \int q\cdot\{g,h\}_p
 &=-\Ren\int\bar h[(\Div q)g+2q_j\partial_jg]\\
 &=\Ren\int\overline{Z_k}\big[
       (\partial_k\Div q)g+(\Div q)\partial_kg
       +2(\partial_kq_j)\partial_jg+2q_j\partial_k\partial_jg\big].
\end{align*}
For the term containing $g\nabla q_m$, set
$Y=E_f(T)^{\frac{1}{2}}X^{\frac{1}{2}}$. Then
\begin{align*}
 \int_0^T\!\int|Zg\nabla q_m|
 &\le\|Z\|_{L^2_tL^6_x}\|g\|_{L^2_tL^{\frac{3}{2}}_x}
                         \|\nabla q_m\|_{L^\infty_tL^6_x}\\
 &\lesssim_{u,C_k}\nu^{-\frac{1}{2}}\nu M_\nu^{\frac{1}{2}}Y=d_\nu^{\frac{1}{2}}Y.
\end{align*}
H\"older's inequality with \eqref{eq:qbounds}, \eqref{eq:Z}, and
\eqref{eq:bernstein-source} also bounds the remaining terms:
\[
 \begin{aligned}
 &\int_0^T\!\int |Z|\big[
     |\nabla g||q_m|+|g||\nabla\Div q|\\
 &\hspace{5em}+|\nabla g|(|\Div q|+|Dq|)+|D^2g||q|\big]\\
 &\qquad\lesssim_{u,C_k}(d_\nu+d_\nu^{\frac{3}{2}}+d_\nu^2)Y.
 \end{aligned}
\]
These bounds also justify the integrations by parts by density:
$q_mg$ and $(\Div q)g+2q\cdot\nabla g$ belong to
$L^2_tW^{1,\frac{6}{5}}_x$, and $Z\in L^2_tL^6_x$.

Decrease $\nu$ so that $d_\nu\le1$. Every displayed power of $d_\nu$ in the preceding estimates is then bounded by $d_\nu^{\frac{1}{2}}$. Summing, using $E_f(T)\le H_f(T)$, and applying Young's inequality gives
\[
 \left|\int_0^T\mathrm{Err}_g(t)\dd t\right|
 \lesssim_{u,C_k} d_\nu^{\frac{1}{2}}E_f(T)^{\frac{1}{2}}X^{\frac{1}{2}}
 \lesssim_{u,C_k} d_\nu^{\frac{1}{2}}[H_f(T)+X].
\]
Together with \eqref{eq:G1-source}, this proves \eqref{eq:source-final}.
\end{proof}

\section{The frequency-localized Morawetz estimate}\label{sec:interaction}
We combine the coercivity and source estimates with
\eqref{eq:forcing-open}. The boundedness of $f$ from
Section~\ref{sec:lp} lets us close the finite-interval inequalities
before invoking the global maximal estimate.

\begin{proof}[Proof of Proposition~\ref{prop:morawetz}]
\emph{Step 1: The finite-interval Morawetz bound.} Fix $\Lambda$ and $c_u$ from Proposition~\ref{prop:weight-estimates}. For each sufficiently small cutoff $\nu$, use its corresponding $h$ and radius $R(t)$. The constants below are independent of $\nu$ and $T$. The integrated identity~\eqref{eq:integrated-identity} holds on every $[0,T]$, and Proposition~\ref{prop:weight-estimates} gives
\[
 |M(t)|\lesssim_u\nu^{-3}.
\]
The lower bound for $\BB_{k(t)}(h(t))+D(t)$ in
Proposition~\ref{prop:weight-estimates}, together with the source
estimate of Proposition~\ref{prop:full-source}, yields
\begin{equation}\label{eq:closure}
 \frac{1}{2}c_uK_a([0,T])
 \le C_u\nu^{-3}
   +C_ud_\nu^{\frac{1}{2}}[\nu^{-3}+K_a([0,T])+H_f(T)].
\end{equation}
All terms are finite: $h\in C([0,T];H^1)$, and $f,F_+$ are bounded. Since $d_\nu\to0$, decrease $\nu_0$, independently of $T$, so that the multiple of $K_a([0,T])$ on the right of~\eqref{eq:closure} is absorbed for every $0<\nu\le\nu_0$.
Consequently,
\begin{equation}\label{eq:morawetz-open}
 K_a([0,T])\le C_u\nu^{-3}+C_ud_\nu^{\frac{1}{2}}H_f(T).
\end{equation}
\emph{Step 2: Close the spacetime estimates.} Substituting into
\eqref{eq:forcing-open} gives, for every $T>0$,
\[
 E_f(T)\le A_\nu+\epsilon_\nu H_f(T),\qquad
 A_\nu=C_u\nu^{-3},\qquad
 \epsilon_\nu=C_u(d_\nu^{\frac{1}{2}}+\eta_\nu^8)\longrightarrow0.
\]
Set $b(T)=\sup_{S\ge T}\frac{E_f(S)}{S}\le F_*^2$.
For $t\in[0,T]$, forward averages of length at most $T$ are
controlled by the maximal function of $f1_{[0,2T]}$, extended by
zero. For $s>T$, Cauchy--Schwarz gives
\[
 \left(\frac{1}{s}\int_t^{t+s}f\right)^2
 \le\frac{E_f(t+s)}{s}\le2b(T).
\]
The $L^2$ maximal theorem \cite[Chapter~I, Section~1]{SteinBook}
therefore yields
$H_f(T)\le CE_f(2T)+2Tb(T)\le C_{\max}Tb(T)$.
Since $b$ is nonincreasing, taking the supremum over $S\ge T$
in the forcing inequality divided by $S$ gives
\[
 b(T)\le\frac{A_\nu}{T}+\epsilon_\nu C_{\max}b(T).
\]
Decrease $\nu_0$ so that
$\epsilon_\nu C_{\max}\le\frac{1}{2}$ for $0<\nu\le\nu_0$.
Since $b(T)<\infty$, absorption gives
$E_f(T)\le Tb(T)\le2A_\nu$. Monotone convergence and the global
maximal theorem now imply
$E_f(\infty)+H_f(\infty)\lesssim_u\nu^{-3}$.
Equation~\eqref{eq:morawetz-open} and monotone convergence prove
\eqref{eq:morawetz-global}.
\end{proof}

\begin{proof}[Proof of Theorem~\ref{thm:main}]
If the theorem fails, Proposition~\ref{prop:reduction} gives a nonzero critical element satisfying \eqref{eq:quasi-soliton}. Fix a sufficiently small $\nu>0$. The concentration bound \eqref{eq:a-bounds} and Proposition~\ref{prop:morawetz} give
\[
 \int_0^\infty N(t)^{-1}\dd t
 \lesssim_u\int_0^\infty\|P_{>\nu}u(t)\|_4^4\dd t
 \lesssim_u\nu^{-3}<\infty,
\]
contradicting \eqref{eq:quasi-soliton}. The reduction and the standard scattering criterion \cite[Section~2]{KM} yield global well-posedness and scattering in both time directions for every datum satisfying \eqref{eq:threshold}.
\end{proof}

\section*{Acknowledgments}
The second author is deeply grateful to his Ph.D. advisor, Prof. Benjamin Dodson, for his generous mentorship and guidance over many years. Their discussions since the author's doctoral studies have helped shape his understanding of and perspective on this line of research. 

Q. Su was supported by the State Key Laboratory of Mathematical Sciences,
Academy of Mathematics and Systems Science, Chinese Academy of Sciences,
and the National Natural Science Foundation of China
(grants 12494541 and 12421001). Z. Zhao was supported by the National Key R\&D Program of China
(grant 2025YFA1018500), the National Natural Science Foundation of China
(grants 12426205 and 12271032), the Beijing Natural Science Foundation
(grant 1262019) and Beijing Institute of
Technology Research Fund Program for Young Scholars.\vspace{1mm}

\paragraph{Data availability.}
No datasets were generated or analyzed for this article.\vspace{1mm}

\paragraph{Competing interests.}
The authors declare no competing interests.\vspace{1mm}

\paragraph{Use of AI}
ChatGPT (OpenAI) made substantial contributions to this work,
including mathematical calculations, the development and checking
of proof arguments, and manuscript drafting. It was also used for
language editing, LaTeX formatting, and bibliographic and consistency
checks. The authors critically reviewed and revised all content and
take full responsibility for the manuscript.

\end{document}